\documentclass[11pt]{amsart}
\usepackage[dvipsnames]{xcolor}
\usepackage{amsfonts,amssymb,amsmath,amscd,amstext}
\usepackage{mathtools}
\usepackage[utf8]{inputenc}
\usepackage{graphicx}
\usepackage{mathrsfs}
\usepackage{changes}
\usepackage{comment}
\usepackage{aligned-overset}
\usepackage{mathdots}
\usepackage{enumerate}
\usepackage[a4paper,top=2cm,bottom=2cm,left=1.4cm,right=1.4cm]{geometry}
\newcommand{\hhh}{{\mathcal{H}}}
\newcommand{\rr}{\mathbb R}
\newcommand{\hh}{{\mathbb{H}}} 
\newcommand{\bb}{{\bf {b}}} 
\renewcommand{\rho}{\varrho} 
\newcommand{\Om}{\Omega}
\newcommand{\eps}{\varepsilon}
\DeclarePairedDelimiter{\abs}{\lvert}{\rvert}
\DeclarePairedDelimiter{\norm}{\lVert}{\rVert}
\DeclareMathOperator{\divv}{div}

\DeclareMathOperator{\spann}{span}
\usepackage[colorlinks=true,linkcolor=teal,citecolor=purple]{hyperref}
\usepackage[noabbrev,capitalize,nameinlink]{cleveref}
\theoremstyle{plain}
\newtheorem{theorem}{Theorem}[section]
\newtheorem{proposition}[theorem]{Proposition}
\newtheorem{lemma}[theorem]{Lemma}

\theoremstyle{definition}

\newtheorem{example}[theorem]{Example}
\theoremstyle{remark}
\newtheorem{remark}[theorem]{Remark}
\crefname{theorem}{Theorem}{Theorems}
\crefname{lemma}{Lemma}{Lemmas}
\crefname{proposition}{Proposition}{Propositions}
\crefname{corollary}{Corollary}{Corollaries}
\crefname{definition}{Definition}{Definitions}
\crefname{remark}{Remark}{Remarks}
\crefname{example}{Example}{Examples}
\title[Transport and flow for horizontal Sobolev contact velocities in Carnot groups]{Transport and flow for horizontal Sobolev contact velocities in Carnot groups}
\author[G.~Somma]{Gianluca Somma}
\address[G.~Somma]{Dipartimento di Matematica ``Tullio Levi-Civita'', Università degli Studi di Padova \protect\newline
\indent Via Trieste 63, 35131 Padova (PD), Italy}
\email{gianluca DOT somm AT phd DOT unipd DOT it}
\author[S.~Verzellesi]{Simone Verzellesi}
\address[S.~Verzellesi]{Department of Decision Sciences and BIDSA, Bocconi University \protect\newline
\indent Via R\"ontgen 1, 20136, Milano (MI), Italy}
\email{simone DOT verzellesi AT unibocconi DOT it}
\author[D.~Vittone]{Davide Vittone}
\address[D.~Vittone]{Dipartimento di Matematica ``Tullio Levi-Civita'', Università degli Studi di Padova \protect\newline
\indent Via Trieste 63, 35131 Padova (PD), Italy}
\email{davide DOT vittone AT unipd DOT it}
\date{\today}

\thanks{\textit{Memberships and funding information.} The authors are members of the Istituto Nazionale di Alta Matematica (INdAM), Gruppo Nazionale per l'Analisi Matematica, la Probabilità e le loro Applicazioni (GNAMPA). The authors are supported by the University of Padova, and received funding through INdAM-GNAMPA 2026 Project \emph{Variational, Geometric, and Analytic Perspectives on Regularity}, CUP E53C25002010001. D.~Vittone is supported by INdAM project {\em VAC\&GMT}}
\numberwithin{equation}{section}
\begin{document}
\begin{abstract}
We establish new well-posedness results for transport and flow equations driven by contact vector fields on Carnot groups. The velocity fields are assumed to have horizontal Sobolev regularity, namely Sobolev regularity only along the horizontal directions determined by the stratified geometry of the group. In the broader sub-Riemannian setting, results of this type were previously known only for Heisenberg groups. Our proof relies on the theory of renormalized solutions as originally introduced by DiPerna and Lions in the Euclidean setting.
%
%
%
%
\end{abstract}
\maketitle

\section{Introduction}
In this paper, we establish new well-posedness results for the \emph{transport equation}
\begin{equation} \label{eq_intro_transport_equation}
\begin{cases}
\displaystyle{\frac{\partial u}{\partial \tau} - \left\langle \bb,\nabla u \right\rangle + cu = 0} & \qquad\text{in $(0,\bar{\tau}) \times \mathbb{G}$}\\
u(0,\cdot)=u_0 & \qquad\text{in $\mathbb{G}$}
\end{cases}
\end{equation}
and for the \emph{flow equation}
\begin{equation}
\label{flow_equation_intro}
\begin{cases}
\displaystyle{\Dot\Phi(\tau,p) =\bb(\tau,\Phi(\tau,p))} & \qquad\text{in $(0,\bar{\tau}) \times \mathbb{G}$}\\
\Phi(0,p)=p & \qquad\text{in $\mathbb{G}$}
\end{cases}
\end{equation}
in the non-Euclidean setting of sub-Riemannian \emph{Carnot groups} $\mathbb G$. In addition to natural growth conditions and a uniform control on its spatial divergence, the velocity field $\bb=\bb(\tau,p)$ has the following properties:
\begin{itemize}
    \item [(i)] \emph{horizontal} Sobolev regularity, i.e.,~Sobolev regularity along a (possibly small) family of directions;
    \item [(ii)] a precise \emph{contact structure}, compatible with the above family of directions.
\end{itemize} 
The original motivation for this theory is to address the Lagrangian problem \eqref{flow_equation_intro} beyond the classical Cauchy-Lipschitz framework (see e.g.~\cite[Chapter 8]{MR2401600}). In the Euclidean setting, this problem was solved in the influential paper \cite{MR1022305} for Sobolev velocity fields, and was later extended to the $BV$ setting in the breakthrough work \cite{MR2096794}. The principle underlying the results of \cite{MR2096794,MR1022305} -- formalized in great generality in \cite{MR2409676,MR3283066} -- is to obtain information at the Lagrangian level from the \emph{Eulerian} viewpoint of \eqref{eq_intro_transport_equation}. More precisely, well-posedness for distributional solutions to  \eqref{eq_intro_transport_equation} yields well-posedness for solutions to \eqref{flow_equation_intro}, for instance in the sense of \emph{regular Lagrangian flows} as introduced in \cite{MR2096794}. 
This Eulerian approach has since led to many further results under suitable regularity and structural assumptions on $\bb$ (see e.g.~\cite{MR4071413,MR1411988,MR2044334,MR2124585}). Purely Lagrangian methods have also been developed (see e.g.~\cite{MR2369485,MR2737390,MR4242824}).

\medskip
A major step toward extending the theory to non-Euclidean geometries was taken in \cite{MR3265963}, which established analogous results for Sobolev vector fields on a broad class of metric measure spaces. Since Carnot groups carry a natural metric measure structure, it is natural to ask whether well-posedness results in this setting follow directly from the theory developed in \cite{MR3265963}. As explained by L.~Ambrosio and the authors in \cite[Section 6]{dpl1}, this is not the case even for \emph{Heisenberg groups}, the simplest non-abelian examples of Carnot groups. Indeed, although Heisenberg groups satisfy the structural requirements imposed in \cite{MR3265963}, the corresponding assumptions on the vector field are too restrictive, reducing the applicability of the theory to the trivial case of the null vector field. This limitation was overcome in \cite{dpl1}, where well-posedness was established for contact vector fields with horizontal Sobolev regularity on Heisenberg groups. The present paper extends these results to the more general setting of Carnot groups.

\medskip
A Carnot group $\mathbb G$ is a connected and simply connected Lie group whose Lie algebra $\mathfrak g$ of left-invariant vector fields admits the non-trivial \emph{stratification}
\begin{equation*}
\mathfrak{g}= V_1 \oplus \dots \oplus V_s, \qquad V_{i+1} = [V_1,V_i] \quad \text{for $i=1,\dots,s-1$,} \qquad [V_1,V_s]=\{0\}.
\end{equation*}
The number $s$ of \emph{layers} in the above decomposition is known as \emph{step} of $\mathbb G$. Fixing a basis of $\mathfrak g$,
\begin{equation}\label{basis_intro_dhdhdhd}
    \left\{X^i_\alpha\right\}_{i,\alpha},\qquad i=1,\dots,s,\qquad\alpha=1,\dots,\dim V_i,
\end{equation}
 \emph{adapted} to the above stratification, uniquely determines a left-invariant Riemannian metric $\langle\cdot,\cdot\rangle$ which makes \eqref{basis_intro_dhdhdhd} orthonormal. Moreover, exponential coordinates identify $\mathbb G$ with $\mathbb R^n$ endowed with a typically non-abelian group law. Under this identification, the Haar measure coincides with the Lebesgue measure. Consequently, if read in exponential coordinates, the intrinsic problem \eqref{eq_intro_transport_equation} takes the form of the usual Euclidean transport problem, so that our results are relevant from both the intrinsic and the Euclidean viewpoint.

\medskip
Both the regularity and the structural assumptions that we impose on $\bb$ are naturally expressed in terms of the so-called \emph{horizontal distribution} $\hhh$, defined by
\begin{equation*}
\hhh_p\coloneqq V_1(p),\qquad p\in\mathbb G.
\end{equation*}
 On the one hand, our regularity assumptions involve derivatives only along horizontal directions. Since $\hhh$ does not, in general, exhaust the whole tangent space, these assumptions are weaker than their Euclidean counterparts. For instance, in the relevant case of Sobolev regularity (see e.g.~\cite{MR4986764}), one typically has
\begin{equation*}
W^{\ell,\theta}(\Om)\subsetneq W_{\mathcal H}^{\ell,\theta}(\Om),
\end{equation*}
where $W_{\mathcal H}^{\ell,\theta}(\Om)$ denotes the corresponding horizontal Sobolev space. In this sense, and as already pointed out in \cite{dpl1}, our results improve upon \cite{MR1022305} by requiring regularity only along horizontal directions.

\medskip
On the other hand, the horizontal distribution also underlies the structural condition imposed on the velocity field. In the smooth setting, a vector field is called \emph{contact} if its flow preserves $\hhh$. This natural class arises in contact geometry (see e.g.~\cite{Libermann,MR266258}), and has been studied thoroughly in the setting of Carnot groups (see e.g.~\cite{MR2395129,MR2917692}). Particular attention has been devoted to Heisenberg and filiform groups, especially in connection with the theory of quasiconformal mappings (see e.g.~\cite{MR788413,MR1317384,MR2016308}). Contact vector fields are characterized by the infinitesimal condition
\begin{equation}\label{intro_infinitesimal}
    [\bb,V_1]_p\subseteq\hhh_p,\qquad\text{for every $p\in \mathbb G$}.
\end{equation}
As \eqref{intro_infinitesimal} is still meaningful under mild regularity assumptions on $\bb$, we adopt it as definition. 
Our main well-posedness result, therefore, applies to (time-dependent) vector fields $\bb$ having horizontal Sobolev regularity in their spatial dependence, and satisfying \eqref{intro_infinitesimal} in the appropriate weak sense.

\medskip
The core idea behind the well-posedness of \eqref{eq_intro_transport_equation} exploits the machinery of \emph{renormalized solutions} introduced in \cite{MR1022305}: when bounded distributional solutions to \eqref{eq_intro_transport_equation} are renormalized solutions (cf.~\Cref{sec_well_posedness}), well-posedness follows by natural growth assumptions and uniform bounds on $\divv \bb$. In turn, the above \emph{renormalization property} follows by adapting the mollification argument introduced in \cite{MR1022305}. More precisely, mollifying \eqref{eq_intro_transport_equation} via \emph{group convolution} with appropriate intrinsic kernels yields
\begin{equation}\label{commutators_intro}
\frac{\partial u_\eps}{\partial \tau}
-\left\langle \bb,\nabla u_\eps\right\rangle
+cu_\eps
=\mathscr C_\eps
\qquad \text{in $(0,\bar{\tau}) \times \mathbb{G}$.}
\end{equation}
The error term $\mathscr C_\eps$, commonly known as \emph{commutator}, quantifies the lack of commutativity between divergence and convolution. The main technical step consists in proving that $\mathscr C_\eps$ converges to $0$ in $L^1_{\mathrm{loc}}$ as $\eps\searrow 0$. As we explain below, the commutator can be decomposed into terms involving difference quotients of the components of $\bb$, of order up to $s$, and suitable remainder terms. The convergence of the former is ensured by horizontal Sobolev regularity, while the contact condition forces the latter to vanish.

\medskip
The paper is organized as follows. In \Cref{sec_preliminaries}, we introduce the setting of Carnot groups and the tools needed throughout the paper, in particular group mollification and horizontal Sobolev spaces.

\medskip
In \Cref{sec_left_right}, we express left-invariant vector fields in terms of right-invariant ones. Namely, if $X^i_\alpha$ is as in \eqref{basis_intro_dhdhdhd} and $\left(X^i_\alpha\right)^r$ is its right-invariant counterpart, we show (see \eqref{eq_notation_multi-index} and \eqref{eq_notation_multi-index_2} for the notation) that
\begin{equation} \label{eq_intro_left_right}
X^{i_0}_{\beta_0} = \left(X^{i_0}_{\beta_0}\right)^r + \sum_{k=1}^{s-i_0} \frac{1}{k!} \sum_{\abs*{I_k} \leq s-i_0} \sum_{\mathcal A_k(I_k)} \sum_{\mathcal B_k(i_0,I_k)} c(I_k,i_0+\abs*{I_k})^{B_k}_{A_k B_{k-1}} x^{I_k}_{A_k} \left(X^{i_0+\abs*{I_k}}_{\beta_k}\right)^r.
\end{equation}
As we shall see, this formula will be crucial for relating the structure of $\mathscr C_\eps$ to the contact condition \eqref{intro_infinitesimal}.

\medskip
In \Cref{sec_contact}, we introduce the relevant class of velocity fields. As already mentioned, a \emph{contact vector field} $\bb$ with \emph{horizontal Sobolev regularity} is defined by requiring that, for some $\theta\in[1,\infty]$,
\begin{equation}\label{def_contact_SObolev_reg_intro}
    \bb=\sum_{i=1}^s\sum_{\alpha=1}^{\dim V_i}b^i_\alpha X^i_\alpha,\qquad b^i_\alpha\in W^{1,\theta}_{\hhh,\mathrm{loc}}(\mathbb G),\qquad [\bb,V_1]_p\subseteq\hhh_p\,\text{ for a.e.~$p\in\mathbb G$.}
\end{equation}
By means of \eqref{def_contact_SObolev_reg_intro}, we show that the first-order Sobolev regularity of the coefficients of $\bb$ improves according to the corresponding layer (see \Cref{rem_iterated_contact_components}). The contact condition in \eqref{def_contact_SObolev_reg_intro} can be equivalently formulated owing to the family of first-order compatibility conditions
\begin{equation} \label{eq_contact_components_first_intro}
X^1_\beta b^k_\gamma = \sum_{\alpha=1}^{\dim V_{k-1}} c(k-1,1)^\gamma_{\alpha \beta} \, b^{k-1}_\alpha\quad \text{ $\beta=1,\dots,\dim V_1$, $k=2,\dots,s$, $\gamma=1,\dots,\dim V_k$},
\end{equation}
where $c(i,j)^\gamma_{\alpha \beta}$ are the so-called \emph{structural constants} of $\mathfrak g$. We show that the improved Sobolev regularity actually upgrades \eqref{eq_contact_components_first_intro} to the higher-order system of iterated compatibility conditions
\begin{equation} \label{eq_intro_iterated_contact_components}
X^{I_k}_{A_k} b^i_{\beta_k} = \sum_{\mathcal B_{k-1}(i-\abs*{I_k},I_{k-1})} \sum_{\alpha=1}^{\dim V_{i-\abs*{I_k}}} c(i,I_k)_{B_{k-1} A_k}^{B_k} b^{i-\abs*{I_k}}_\alpha.
\end{equation}
(see \Cref{prop_contact_components} and  \Cref{peodneicjecneficnecijn}). The identities in \eqref{eq_intro_iterated_contact_components} will play an essential role in \Cref{sec_well_posedness}, as they allow us to get rid of the aforementioned remainder terms arising in the commutator analysis.
 The section concludes with a few structural remarks and further insights into the class of contact vector fields.

 \medskip
In \Cref{sec_difference_quotients}, we investigate the asymptotic behavior of (higher-order) difference quotients for functions with horizontal Sobolev regularity. More precisely, if $f\in W^{\ell,\theta}_\hhh$ and $\{\delta_\eps\}_{\eps>0}$ denotes the appropriate family of \emph{intrinsic dilations} of $\mathbb G$, we prove that
\begin{equation}\label{eq_intro_diff_quot}
\frac{f\left(p\cdot\delta_\eps(w)\right)-P^\ell_p f\left(\delta_\eps(w)\right)}{\eps^\ell}
\xrightarrow[\eps\searrow0]{L^\theta}0,
\end{equation}
where $P^\ell_\cdot f$ is the intrinsic Taylor polynomial of $f$ of \emph{homogeneous order} $\ell$. Thus, \eqref{eq_intro_diff_quot} provides an $L^\theta$-Taylor expansion adapted to the homogeneous structure of $\mathbb G$, thereby furnishing the second key ingredient in the forthcoming commutator analysis.

 \medskip
In \Cref{sec_well_posedness}, we prove the main results of the paper. As stressed above, the key step is the analysis of the commutator $\mathscr C_\eps$ in \eqref{commutators_intro}, which we carry out in \Cref{sec_refgogogog}.
Using \eqref{eq_intro_left_right} and performing some algebraic manipulations, one finds that $\mathscr C_\eps$ can be written as combination of terms involving the difference quotients of the components of $\bb$, namely
\begin{equation} \label{eq_intro_diff_quot_vf}
\int_{\mathbb{G}} \frac{b^i_\alpha(p \cdot \delta_\eps(w))-P^{i-1}_p b^i_\alpha (\delta_\eps(w))}{\eps^i} u(p \cdot \delta_\eps(w)) X^i_\alpha \rho(w),dw,
\end{equation}
and remainder terms of the form
\begin{equation} \label{eq_intro_remainders}
\left[X^{I_k}_{A_k} b^i_{\beta_k}(p) - \sum_{\mathcal B_{k-1}(i-\ell,I_{k-1})} \sum_{\alpha=1}^{\dim V_{i-\ell}} c(i,I_k)_{B_{k-1} A_k}^{B_k} b^{i-\ell}_\alpha(p)\right] \int_{\mathbb{G}} \delta_\eps\left(w^{I_k}_{A_k}\right) u(p \cdot \delta_\eps(w)) \eps^{-i} X^i_{\beta_k} \rho(w) \,dw.
\end{equation}
On the one hand, \eqref{eq_intro_diff_quot} and the improved Sobolev regularity provided by \eqref{def_contact_SObolev_reg_intro} imply that \eqref{eq_intro_diff_quot_vf} converges to $0$ in $L^1_{\mathrm{loc}}$. On the other hand, \eqref{eq_intro_iterated_contact_components} shows that \eqref{eq_intro_remainders} vanishes identically. Consequently, the renormalization property and the well-posedness of \eqref{eq_intro_transport_equation} follow essentially as in \cite{dpl1}, under the above-mentioned growth conditions and controllability assumptions on $\bb$. Once the Eulerian problem is settled, well-posedness for the Lagrangian problem \eqref{flow_equation_intro} is deduced by means of the general theory developed in \cite{MR2409676,MR3283066} (see \Cref{section_wpr}).

\medskip
Finally, in \Cref{sec_appendix} we collect some additional results which may be of independent interest, including an inductive representation formula for difference quotients and a quantitative mollification approximation which follows at once by the results of \Cref{sec_difference_quotients}.

\subsection*{Acknowledgments}
The authors thank Luigi Ambrosio, Enrico Le Donne, and Alessandro Ottazzi for many stimulating and enlightening conversations.

\section{Preliminaries} \label{sec_preliminaries}
\subsection{Main notation}
We denote by $\mathbb{N}_+$ the set of nonzero natural numbers. We also write $\infty=+\infty$. If $\bar{\tau} \in (0,\infty]$, the notation $[0,\bar{\tau}]$ stays for the usual closed interval if $\bar{\tau}<\infty$ or $[0,\infty)$ if $\bar{\tau}=\infty$. Given open sets $A,\Om$, when $\overline{A}$ is a compact subset of $\Om$, we write $A\Subset\Om$. If $\theta \in [1,\infty]$, we denote its H\"older conjugate by $\theta'\in [1,\infty]$, i.e., $\frac{1}{\theta}+\frac{1}{\theta'}=1$. 
If $f$ is differentiable at a point $q$, we denote by $df_q$ its differential at $q$.

\subsection{Carnot groups}
A Lie algebra $\mathfrak{g}$ is said to be \emph{stratified} if it admits a \emph{stratification}, i.e., there exist nontrivial linear subspaces $V_1, \dots, V_s$ of $\mathfrak{g}$ such that
\begin{equation}\label{defcarnotgroup}
\mathfrak{g}= V_1 \oplus \dots \oplus V_s, \qquad V_{i+1} = [V_1,V_i] \quad \text{for $i=1,\dots,s-1$,} \qquad [V_1,V_s]=\{0\}.
\end{equation}
We stress that $[V_i,V_j]$ denotes the linear span of vectors of the form $[X,Y]$ with $X \in V_i$ and $Y \in V_j$. Recall that the Jacobi identity
\begin{equation} \label{eq_Jacobi}
[X,[Y,Z]]+[Y,[Z,X]]+[Z,[X,Y]]=0 \qquad \text{for every $X,Y,Z \in\mathfrak{g}$}
\end{equation}
holds. Setting $V_k=\{0\}$ if $k>s$, \eqref{eq_Jacobi} implies that $[V_i,V_j] \subseteq V_{i+j}$ for every $i \neq j$. If $(\mathbb{G},\cdot)$ is a Lie group, we denote by $\mathfrak{g}$ its Lie algebra. Moreover, we define
\begin{equation*}
L_p(q) = p \cdot q, \qquad R_p(q) = q \cdot p \qquad \text{for every $p,\,q \in \mathbb{G}$}
\end{equation*}
to be the left and right translations, respectively. Left-invariant vector fields are complete (see \cite[Theorem 9.18]{MR2954043}). Accordingly, if $Z \in \mathfrak{g}$ and $\gamma$ is its integral curve starting from the identity element $e$, we define the exponential map $\exp : \mathfrak{g} \to \mathbb{G}$ by
\begin{equation*}
\exp(Z) \coloneqq \gamma(1).
\end{equation*}
We say that $(\mathbb{G},\cdot)$ is a \emph{stratified Lie group} if it is connected, simply connected, and its Lie algebra $\mathfrak{g}$ is stratified. Let $\mathfrak{g}=V_1 \oplus \dots \oplus V_s$ be a stratification and set $n_i \coloneqq \dim V_i$. Set also $n \coloneqq \dim \mathfrak{g} = \dim \mathbb{G}$. Let $\left\{X^i_\alpha\right\}_{\alpha=1,\dots,n_i}$ be a basis for $V_i$ for every $i=1,\dots,s$. Since the exponential map of a stratified Lie group is a diffeomorphism (see \cite[Theorem 3.6.2]{MR746308}), we can define a system of \emph{graded coordinates} associated with the basis $\left\{X^i_\alpha\right\}_{\substack{i=1,\dots,s \\ \alpha=1,\dots,n_i}}$ for $\mathfrak{g}$:
\begin{equation*}
\rr^n \ni \left(x^1_1,\dots,x^1_{n_1},x^2_1,\dots,x^2_{n_2},\dots,x^s_1,\dots,x^s_{n_s}\right) \mapsto \exp\left(\sum_{i=1}^s \sum_{\alpha=1}^{n_i} x^i_\alpha X^i_\alpha\right) \in \mathbb{G}.
\end{equation*}
In these coordinates, the identity element is $0$ and the inverse element of $p \in \mathbb{G}$ is $p^{-1}=-p$. We denote by $\log : \mathbb{G} \to \mathfrak{g}$ the inverse of $\exp$. We fix a system of graded coordinates and denote by $X^i_\alpha$ and $x^i_\alpha(p)$ (or simply $x^i_\alpha$ if there is no ambiguity), with $i=1,\dots,s$ and $\alpha=1,\dots,n_i$, the $\alpha$-th vector field of the basis of $V_i$ and the corresponding graded coordinate of $p \in \mathbb{G}$. Sometimes, we will adopt the notation
\begin{equation*}
p=\left(x^1,\dots,x^s\right), \qquad x^i=\left(x^i_1,\dots,x^i_{n_i}\right) \in \rr^{n_i}.
\end{equation*}
We equip $\mathbb{G}$ with the Riemannian metric $\langle \cdot, \cdot \rangle$ that makes the above basis orthonormal. Accordingly, the Riemannian gradient reads as
\begin{equation*}
\nabla f =\sum_{i=1}^s\sum_{\alpha=1}^{n_i}\left(X^i_\alpha f\right) X^i_\alpha.
\end{equation*}
We denote by $\left(X^i_\alpha \right)^r$ the right-invariant vector field corresponding to $X^i_\alpha$, that is,
\begin{equation*}
\left(X^i_\alpha \right)^r(p)=\left(dR_p\right)_e\left(X^i_\alpha(e)\right) \qquad \text{for every $p \in \mathbb{G}$.}
\end{equation*}
Moreover, we define the \emph{structural constants} $c(i,j)^\gamma_{\alpha \beta} \in \rr$ by
\begin{equation} \label{eq_structural_constants}
\left[X^i_\alpha,X^j_\beta\right] \coloneqq \sum_{\gamma=1}^{n_{i+j}} c(i,j)^\gamma_{\alpha \beta} X^{i+j}_\gamma
\end{equation}
for every $i,j=1,\dots,s$ such that $i+j \leq s$, $\alpha=1,\dots,n_i$, $\beta=1,\dots,n_j$. Observe that
\begin{equation} \label{eq_skew_structural_constants}
c(i,j)_{\alpha \beta}^\gamma = - c(j,i)_{\beta \alpha}^\gamma.
\end{equation}
We endow $\mathbb{G}$ with a \emph{homogeneous structure} (see \cite{BonLanUgu}). Precisely, for any $\lambda>0$, we define the \emph{intrinsic dilations} by setting $\delta_\lambda Z \coloneqq \lambda^i Z$ if $Z \in V_i$ and extending it by linearity on the whole algebra $\mathfrak{g}$. Then, considering their image through the exponential map, we get a one-parameter group of dilations on $\mathbb{G}$, which in graded coordinates are expressed as follows:
\begin{equation*}
\delta_\lambda(p) \coloneqq \left(\lambda x^1_1,\dots,\lambda x^1_{n_1}, \lambda^2 x^2_1,\dots,\lambda^2 x^2_{n_2},\dots,\lambda^s x^s_1,\dots,\lambda^s x^s_{n_s}\right) \qquad \text{for every $\lambda>0, \, p=\left(x^1,\dots,x^s\right) \in \mathbb{G}$.}
\end{equation*}
The \emph{horizontal distribution} $\hhh$ is defined by
\begin{equation*}
\hhh_p \coloneqq V_1(p) \qquad \text{for every $p\in \mathbb{G}$.}
\end{equation*}
We endow $\mathbb{G}$ with the so-called \emph{Carnot-Carathéodory distance} $d$ (see \cite{MR3587666}): given $p,q\in\mathbb{G}$, set
\begin{equation*}
    d(p,q) \coloneqq \inf\left\{\int_0^1\sqrt{\left\langle\dot\gamma,\dot\gamma\right\rangle}\,d\tau\,:\, \gamma:[0,1]\to\mathbb{G}\text{ is absolutely continuous, horizontal, } \gamma(0)=p,\,\gamma(1)=q\right\},
\end{equation*}
where an absolutely continuous curve is \emph{horizontal} if $\dot\gamma(\tau)\in\hhh_{\gamma(\tau)}$ for a.e.~$\tau$. By Chow's theorem, $d$ is finite. The metric space $(\mathbb{G},d)$ is called a \emph{Carnot group}. We indicate by $B(p,r) \coloneqq \{q \in\mathbb{G} : d(p,q)<r\}$ the open ball with center $p \in \mathbb{G}$ and radius $r>0$.
The Carnot-Carathéodory distance is compatible with the group and homogeneous structures:
\begin{equation*} 
\begin{split}
    d(w \cdot p, w \cdot q) & = d(p,q),\\
    d(\delta_\lambda(p),\delta_\lambda(q)) & = \lambda d(p,q)
\end{split}
\qquad \text{for every $p,q,w \in \mathbb{G},\,\lambda>0$.}
\end{equation*}
The Lebesgue measure $|\cdot|$ is a Haar measure:
\begin{equation} \label{eq_Lebesgue_Haar}
\begin{split}
    \abs*{L_p(E)}&=\abs*{E},\\
    \abs*{R_p(E)}&=\abs*{E}
\end{split}
\qquad \text{for every $E \subseteq\mathbb{G}$ measurable, $p \in\mathbb{G}$.}
\end{equation}
Furthermore,
\begin{equation} \label{eq_dilation_Lebesgue}
\abs*{\delta_\lambda(E)}=\lambda^Q\abs*{E} \qquad \text{for every $E \subseteq\mathbb{G}$ measurable, $\lambda>0$,}
\end{equation}
where $Q \coloneqq \sum_{i=1}^s i n_i$ is the \emph{homogeneous dimension} of $(\mathbb{G},\cdot)$ (see \cite{MR3587666}). Using arguments identical to those in \cite[Lemma 2.1]{dpl1} and \cite[Lemma 2.2]{dpl1}, the following simple properties hold.
\begin{lemma}\label{lemscambiarederivateeinversione}
Define $\iota: \mathbb{G}\to\mathbb{G}$ by $\iota(q)=q^{-1}$. Then
\begin{equation*} 
X^i_\alpha(\varphi \circ \iota)=-\left(\left(X^i_\alpha \right)^r \varphi\right) \circ \iota \qquad \text{for every } \varphi \in C^1(\mathbb{G}),\, i=1,\ldots,s, \, \alpha=1,\ldots,n_i.
\end{equation*}
In particular, if $\varphi=\varphi\circ\iota$, then $X^i_\alpha\varphi=-\left(\left(X^i_\alpha \right)^r\varphi\right)\circ\iota$.
\end{lemma} 
\begin{lemma} \label{lem_continuity_translations}
Let $\theta \in [1,\infty)$ and $f \in L^\theta(\mathbb{G})$. Then $\displaystyle \lim_{q\to 0}\norm*{f \circ R_q - f}_{L^\theta(\mathbb{G})} =0$.
\end{lemma}
We denote by $\mathbb{X}(\mathbb{G})$ the class of locally integrable vector fields in $\mathbb{G}$, and we adopt the notation
\begin{equation} \label{eq_notation_vector_field}
\bb=\sum_{i=1}^s \sum_{\alpha=1}^{n_i} b^i_\alpha X^i_\alpha \qquad \text{for every $\bb \in \mathbb{X}(\mathbb{G})$.}    
\end{equation}
Since the frame $\left\{X^i_\alpha\right\}_{\substack{i=1,\dots,s \\ \alpha=1,\dots,n_i}}$ is orthonormal and its coefficients matrix with respect to the canonical basis of $\rr^n$ has determinant equal to $1$ (cf.~\cite[Corollary 1.3.19]{BonLanUgu}), the Riemannian volume coincides with the Lebesgue measure. Hence, we define the divergence of a vector field $\bb \in \mathbb{X}(\mathbb{G})$ as the distribution $\divv \bb$ such that
\begin{equation*}
\langle \divv \bb, \varphi \rangle \coloneqq -\int_{\mathbb{G}} d\varphi(\bb) \,dp \qquad \text{for every $\varphi \in C^\infty_c(\mathbb{G})$.}
\end{equation*}
Thanks to \eqref{eq_Lebesgue_Haar}, one can prove that left and right-invariant vector fields are divergence-free.

\subsection{Group convolution}\label{subsec_groupconv}
We refer to \cite{MR0657581} for more details. Given $f,g \in L^1_{\mathrm{loc}}(\mathbb{G})$, their \emph{group convolution} $f*g$ is defined by
\begin{equation*}
(f*g)(p) \coloneqq \int_{\mathbb{G}} f(q)g\left(q^{-1} \cdot p\right) \, dq \overset{\eqref{eq_Lebesgue_Haar}}{=} \int_{\mathbb{G}} f\left(p \cdot q^{-1}\right)g(q) \, dq \qquad \text{for every $p \in \mathbb{G}$,}
\end{equation*}
provided that the integrals converge. If the domain of $f$ and $g$ is an arbitrary open set $\Omega \subseteq \mathbb{G}$, we can still define $f*g$ by extending them to be $0$ outside $\Omega$.
Fix a mollifier $\rho \in C_c^\infty(\mathbb{G})$ such that
\begin{equation}\label{eq_mollificatori}
\rho \geq 0, \quad \rho(p)=\rho\left(p^{-1}\right), \quad \int_{\mathbb{G}} \rho\,dp = 1 \quad \text{and} \quad \rho \equiv 0\text{ in } B(0,1)^c
\end{equation}
and set
\begin{equation} \label{eq_def_mollifier}
\rho_\varepsilon(p)=\frac{1}{\varepsilon^Q}\rho \left(\delta_{\frac{1}{\varepsilon}}(p) \right) \qquad \text{for every $\eps>0$,\, $p \in \mathbb{G}$.}
\end{equation}
Observe that
\begin{equation} \label{eq_derivative_mollification}
\left(X^i_\alpha\right)^r\rho_\varepsilon(p) =\frac{1}{\varepsilon^{Q+i}} \left(\left(X^i_\alpha\right)^r \rho\right) \left(\delta_{\frac{1}{\varepsilon}}(p) \right) \qquad \text{for every $i=1,\ldots,s$,\, $\alpha=1,\ldots,n_i$,\, $p\in\mathbb{G}$.}
\end{equation}
The following standard convergence properties of group mollification hold.
\begin{proposition} \label{prop_group_mollification}
Let $\theta \in [1,\infty)$. Let $u \in L^\theta(\mathbb{G})$ and set $u_\varepsilon=u * \rho_\varepsilon$ for every $\varepsilon>0$. Then
\begin{enumerate}[(i)]
\item $u_\eps \xrightarrow[]{L^\theta} u$ as $\varepsilon \searrow 0$,
\item $\norm*{u_\eps}_{L^\theta(\mathbb{G})} \leq \norm*{u}_{L^\theta(\mathbb{G})}$ for every $\varepsilon>0$.
\end{enumerate}
\end{proposition}

\subsection{Horizontal Sobolev spaces}
We refer to \cite{MR494315,MR4544986} for further information. Fix an open subset $\Om\subseteq \mathbb{G}$ and $u\in L^1_{\mathrm{loc}}(\Om)$. If $i=1,\dots,s$ and $\alpha=1,\dots,n_i$, the distribution $X^i_\alpha u$ is defined by 
\begin{equation*}
\left\langle X^i_\alpha u,\varphi\right\rangle \coloneqq -\int_\Om u X^i_\alpha \varphi\,dp \qquad \text{for every $\varphi\in C^\infty_c(\Om)$}.
\end{equation*}
If $\theta\in[1,\infty]$, we define the \emph{horizontal Sobolev space} $W^{1,\theta}_\hhh(\Om)$ by
\begin{equation*}
W^{1,\theta}_\hhh(\Om)\coloneqq \left\{u\in L^\theta(\Om)\,:\,X^1_\alpha u\in L^\theta(\Om)\text{ for every $\alpha=1,\dots,n_1$}\right\}.
\end{equation*}
The spaces $W^{1,\theta}_{\hhh,\mathrm{loc}}(\Om)$ and $W^{\ell,\theta}_\hhh(\Om)$, for $\ell=2,3,\ldots$, are defined accordingly. It is easy to see that the vector space $W^{\ell,\theta}_\hhh(\Om)$, endowed with the norm
\[
\Vert u\Vert_{W^{\ell,\theta}_\hhh(\Om)}\coloneqq\Vert u\Vert_{L^\theta(\Om)}+ \sum_{\alpha=1}^{n_1} \left\|X^1_\alpha u\right\|_{L^\theta(\Om)}+\dots+\sum_{\alpha_1,\ldots,\alpha_\ell=1}^{n_1} \left\|X^1_{\alpha_1}\ldots X^1_{\alpha_\ell} u \right\|_{L^\theta(\Om)},
\]
is a Banach space for every $1\leq \theta \leq \infty$ and $\ell \in \mathbb N_+$. Moreover, the following Meyers-Serrin-type approximation result holds (see \cite[Theorem A.2]{MR1404326}). 
\begin{theorem} \label{meyersserrin}
    Let $\Om\subseteq \mathbb{G}$ be an open set. Let $\ell \in\mathbb N_+$ and $\theta \in [1,\infty)$. Let $u\in W^{\ell,\theta}_\hhh(\Om)$. There exists a sequence $\{u_h\}_{h\in\mathbb N}\subseteq C^\infty(\Om)\cap W^{\ell,\theta}_\hhh(\Om)$ such that 
    \begin{equation*}
        \lim_{h\to\infty }\|u_h-u\|_{W^{\ell,\theta}_\hhh(\Om)}=0.
    \end{equation*}
\end{theorem}
Observe that, if $\bb$ is as in \eqref{eq_notation_vector_field} and $b^i_\alpha \in W^{1,\theta}_{\hhh,\mathrm{loc}}(\Om)$ for every $i=1,\dots,s$ and $\alpha=1,\dots,n_i$, then
\begin{equation} \label{eq_divergence}
\divv \bb = \sum_{i=1}^s \sum_{\alpha=1}^{n_i} X^i_\alpha b^i_\alpha.
\end{equation}

\section{Left and right-invariant vector fields} \label{sec_left_right}
In this section, we express left-invariant vector fields with respect to the right-invariant frame, providing some relevant examples. The key ingredient is the differential at the identity element $e$ of the conjugation map: for any $p \in \mathbb{G}$, set $C_p \coloneqq L_p \circ R_{p^{-1}}$ and $\mathrm{Ad}_p \coloneqq \left(d C_p\right)_e$. Denote by $\mathrm{End}(\mathfrak{g})$ the set of all endomorphisms of $\mathfrak{g}$ and define the operator $\mathrm{ad}_Z: \mathfrak{g} \to \mathrm{End}(\mathfrak{g})$ by $\mathrm{ad}_Z X \coloneqq [Z,X]$ for every $X,Z \in \mathfrak{g}$. Then, it is known that (see \cite[Proposition 1.91]{MR1920389})
\begin{equation} \label{eq_adjoint_exponential}
\mathrm{Ad}_{\exp(Z)} = e^{\mathrm{ad}_Z},
\end{equation}
where $e^A \coloneqq \sum_{k=0}^\infty A^k/k! \in \mathrm{End}(\mathfrak{g})$ for every $A \in \mathrm{End}(\mathfrak{g})$. We introduce a specific notation to simplify the next computations. For all $k \in \mathbb{N}_+$, we denote a generic $k$-multi-index by $I_k=(i_1,\dots,i_k) \in \{1,\dots,s\}^k$, and we set $\abs*{I_k} \coloneqq i_1+\dots+i_k$. Moreover, given $i=1,\dots,s-1$, we define the sets
\begin{equation} \label{eq_notation_multi-index}
\begin{split}
    \mathcal A_k(I_k) & = \left\{(\alpha_1,\dots,\alpha_k) \in \mathbb{N}^k : 1 \leq \alpha_j \leq n_{i_j} \right\},\\
    \mathcal B_k(i,I_k) & = \left\{(\beta_1,\dots,\beta_k) \in \mathbb{N}^k : 1 \leq \beta_j \leq n_{i+i_1+\dots+i_j} \right\}, \qquad \text{ provided that $i+\abs*{I_k} \leq s$.}
\end{split}
\end{equation}
We adopt the notation
\begin{equation*}
\sum_{\mathcal A_k(I_k)} = \sum_{(\alpha_1,\dots,\alpha_k) \in \mathcal A_k}, \qquad \sum_{\mathcal B_k(i,I_k)} = \sum_{(\beta_1,\dots,\beta_k) \in \mathcal B_k(i,I_k)}.
\end{equation*}
\begin{lemma}
Let $i_0=1,\dots,s-1$ and $\beta_0=1,\dots,n_{i_0}$. If $Z=\displaystyle \sum_{i=1}^s \sum_{\alpha=1}^{n_i} x^i_\alpha X^i_\alpha \in \mathfrak{g}$ and $k=1,\dots,s-i_0$, then
\begin{equation} \label{eq_iterated_adjoint}
\begin{split}
\left(\mathrm{ad}_Z\right)^k X^{i_0}_{\beta_0} & = \sum_{\abs*{I_k} \leq s-i_0} \sum_{\mathcal A_k(I_k)} \sum_{\mathcal B_k(i_0,I_k)} c(i_1,i_0)^{\beta_1}_{\alpha_1 \beta_0} \dots c(i_k,i_0+i_1+\dots+i_{k-1})^{\beta_k}_{\alpha_k \beta_{k-1}} x^{i_1}_{\alpha_1} \dots x^{i_k}_{\alpha_k} X^{i_0+\abs*{I_k}}_{\beta_k}.
\end{split}
\end{equation}
\end{lemma}
\begin{proof}
We prove the formula by induction on $k$. A direct computation shows that
\begin{equation*}
\mathrm{ad}_Z X^{i_0}_{\beta_0}=\left[Z,X^{i_0}_{\beta_0} \right] = \sum_{i_1=1}^{s-i_0} \sum_{\alpha_1=1}^{n_{i_1}} x^{i_1}_{\alpha_1} \left[X^{i_1}_{\alpha_1}, X^{i_0}_{\beta_0} \right] \overset{\eqref{eq_structural_constants}}{=} \sum_{i_1=1}^{s-i_0} \sum_{\alpha_1=1}^{n_{i_1}} \sum_{{\beta_1}=1}^{n_{i_0+i_1}} c(i_1,i_0)^{\beta_1}_{\alpha_1 \beta_0} x^{i_1}_{\alpha_1} X^{i_0+i_1}_{\beta_1},
\end{equation*}
which is \eqref{eq_iterated_adjoint} with $k=1$. Then, assume $k>1$ and that the thesis holds for $k-1$. Hence, the inductive hypothesis and the base case ensure
\begin{equation*}
\begin{split}
& \left(\mathrm{ad}_Z\right)^k X^{i_0}_{\beta_0} = \Bigl[\underbrace{Z,\Bigl[Z,\dots,\Bigl[Z}_{k},X^{i_0}_{\beta_0} \Bigr]\dots \Bigr]\Bigr]\\
& = \sum_{\abs*{I_{k-1}} \leq s-i_0} \sum_{\mathcal A_{k-1}(I_{k-1})} \sum_{\mathcal B_{k-1}(i_0,I_{k-1})} c(i_1,i_0)^{\beta_1}_{\alpha_1 \beta_0} \dots c(i_{k-1},i_0+\dots+i_{k-2})^{\beta_{k-1}}_{\alpha_{k-1} \beta_{k-2}} x^{i_1}_{\alpha_1} \dots x^{i_{k-1}}_{\alpha_{k-1}} \mathrm{ad}_ZX^{i_0+\abs*{I_{k-1}}}_{\beta_{k-1}}\\
\overset{\eqref{eq_structural_constants}}&{=} \sum_{\abs*{I_k} \leq s-i_0} \sum_{\mathcal A_k(I_k)} \sum_{\mathcal B_k(i_0,I_k)} c(i_1,i_0)^{\beta_1}_{\alpha_1 \beta_0} \dots c(i_k,i_0+\dots+i_{k-1})^{\beta_k}_{\alpha_k \beta_{k-1}} x^{i_1}_{\alpha_1} \dots x^{i_k}_{\alpha_k} X^{i_0+\abs*{I_k}}_{\beta_k},
\end{split}
\end{equation*}
which is the thesis.
\end{proof}
From now on, whenever $k \in \mathbb{N}_+$, $A_k=(\alpha_1,\dots,\alpha_k) \in \mathcal A_k(I_k)$ and $B_k=(\beta_1,\dots,\beta_k)=(B_{k-1},\beta_k) \in \mathcal B_k(i,I_k)$, we may adopt the compact notation
\begin{equation} \label{eq_notation_multi-index_2}
\begin{split}
    x^{I_k}_{A_k} & \coloneqq x^{i_1}_{\alpha_1} \dots x^{i_k}_{\alpha_k} \qquad \text{for every $\left(x^1,\dots,x^s\right) \in \mathbb{G}$},\\
    X^{I_k}_{A_k} & \coloneqq X^{i_1}_{\alpha_1} \dots X^{i_k}_{\alpha_k},\\
    c(I_k,i)^{B_k}_{A_k B_{k-1}} & \coloneqq c(i_1,i-i_k-\dots-i_1)^{\beta_1}_{\alpha_1 \alpha} \dots c(i_k,i-i_k)^{\beta_k}_{\alpha_k \beta_{k-1}} \qquad \text{for every $\abs*{I_k} < i$, $\alpha=1,\dots n_i$,}\\
    c(i,I_k)^{B_k}_{B_{k-1} A_k} & \coloneqq c(i-i_k-\dots-i_1,i_1)^{\beta_1}_{\alpha \alpha_1} \dots c(i-i_k,i_k)^{\beta_k}_{\beta_{k-1} \alpha_k} \qquad \text{for every $\abs*{I_k} < i$, $\alpha=1,\dots n_i$.}
\end{split}
\end{equation}
\begin{theorem}
Let $i_0=1,\dots,s-1$ and $\beta_0=1,\dots,n_{i_0}$. Then
\begin{equation} \label{eq_right_vector_fields}
X^{i_0}_{\beta_0} = \left(X^{i_0}_{\beta_0}\right)^r + \sum_{k=1}^{s-i_0} \frac{1}{k!} \sum_{\abs*{I_k} \leq s-i_0} \sum_{\mathcal A_k(I_k)} \sum_{\mathcal B_k(i_0,I_k)} c(I_k,i_0+\abs*{I_k})^{B_k}_{A_k B_{k-1}} x^{I_k}_{A_k} \left(X^{i_0+\abs*{I_k}}_{\beta_k}\right)^r.
\end{equation}
Instead, if $i_0=s$, then $X^s_{\beta_0}=\left(X^s_{\beta_0}\right)^r$.
\end{theorem}
\begin{proof}
For any $p \in \mathbb{G}$, the identity $L_p = R_p\circ C_p$ holds. Therefore, the left-invariance of $X^{i_0}_{\beta_0}$ implies
\begin{equation*}
X^{i_0}_{\beta_0}(p)=\left(dL_p\right)_e\left(X^{i_0}_{\beta_0}(e)\right)=\left(dR_p\right)_e \left(\mathrm{Ad}_p X^{i_0}_{\beta_0}(e)\right).
\end{equation*}
Set $Z=\log(p)$. Then,
\begin{equation*}
    X^{i_0}_{\beta_0}(p) \overset{\eqref{eq_adjoint_exponential}}{=} \left(dR_p\right)_e \left(e^{\mathrm{ad}_Z} X^{i_0}_{\beta_0}(e)\right)= \sum_{k=0}^{s-i_0} \frac{1}{k!} \left(dR_p\right)_e \left(\left(\mathrm{ad}_Z\right)^k X^{i_0}_{\beta_0}(e)\right),
\end{equation*}
which allows us to conclude if $i_0=s$. If, instead, $i_0<s$, then
\begin{equation*}
\begin{split}
    X^{i_0}_{\beta_0}(p) \overset{\eqref{eq_iterated_adjoint}}&{=}  \left(dR_p\right)_e\left(X^{i_0}_{\beta_0}(e)\right) + \sum_{k=1}^{s-i_0} \frac{1}{k!} \sum_{\abs*{I_k} \leq s-i_0} \sum_{\mathcal A_k(I_k)} \sum_{\mathcal B_k(i_0,I_k)} c(I_k,i_0+\abs*{I_k})^{B_k}_{A_k B_{k-1}} x^{I_k}_{A_k}(p) \left(dR_p\right)_e \left(X^{i_0+\abs*{I_k}}_{\beta_k}(e)\right)\\
    & = \left(X^{i_0}_{\beta_0}\right)^r(p) + \sum_{k=1}^{s-i_0} \frac{1}{k!} \sum_{\abs*{I_k} \leq s-i_0} \sum_{\mathcal A_k(I_k)} \sum_{\mathcal B_k(i_0,I_k)}c(I_k,i_0+\abs*{I_k})^{B_k}_{A_k B_{k-1}} x^{I_k}_{A_k}(p) \left(X^{i_0+\abs*{I_k}}_{\beta_k}\right)^r(p),
\end{split}
\end{equation*}
which is the thesis.
\end{proof}
\begin{example} \label{ex_Heisenberg}
The Lie algebra of the $n$-th \emph{Heisenberg group} (see e.g.~\cite{BonLanUgu}) is generated by the vector fields
\begin{equation*}
X^1_\alpha =\partial_{x^1_\alpha}+2x^1_{\alpha+n} \partial_{x^2_1}, \qquad 
X^1_{\alpha+n} =\partial_{x^1_{\alpha+n}}-2x^1_\alpha \partial_{x^2_1}, \qquad X^2_1 =\partial_{x^2_1} \qquad \text{for every $\alpha=1,\ldots,n$.}
\end{equation*}
In this case, $s=2$. We have $c(1,1)^1_{\alpha \beta}=0$ if $\alpha<\beta$ and $\beta \neq \alpha+n$, and $c(1,1)^1_{\alpha,\alpha+n}=-4$ for every $\alpha=1,\dots,n$. Recall that $\left(X^2_1\right)^r=X^2_1$. Regarding $X^1_\alpha$ and $X^1_{\alpha+n}$, we obtain
\begin{equation*}
\begin{split}
    X^1_\alpha \overset{\eqref{eq_right_vector_fields}}&{=} \left(X^1_\alpha\right)^r +\sum_{\alpha_1=1}^{2n} c(1,1)^1_{\alpha_1,\alpha} x^1_{\alpha_1} X^2_1 = \left(X^1_\alpha\right)^r + c(1,1)^1_{\alpha+n,\alpha}  x^1_{\alpha+n} X^2_1 = \left(X^1_\alpha\right)^r + 4x^1_{\alpha+n} X^2_1,\\
    X^1_{\alpha+n} \overset{\eqref{eq_right_vector_fields}}&{=} \left(X^1_{\alpha+n}\right)^r +\sum_{\alpha_1=1}^{2n} c(1,1)^1_{\alpha_1,\alpha+n} x^1_{\alpha_1} X^2_1 = \left(X^1_{\alpha+n}\right)^r + c(1,1)^1_{\alpha,\alpha+n}  x^1_\alpha X^2_1 = \left(X^1_{\alpha+n}\right)^r - 4x^1_\alpha X^2_1.
\end{split}
\end{equation*}
\end{example}
\begin{example}
The Lie algebra of the \emph{Engel group} (see e.g.~\cite{BonLanUgu}) is generated by the vector fields
\begin{equation*}
\begin{split}
X^1_1 = \partial_{x^1_1}+\frac{x^1_2}{2}\partial_{x^2_1}+\left(\frac{x^2_1}{2}-\frac{x^1_1 x^1_2}{12}\right)\partial_{x^3_1}, \quad 
X^1_2 =\partial_{x^1_2}-\frac{x^1_1}{2}\partial_{x^2_1}+\frac{(x_1^1)^2}{12}\partial_{x^3_1}, \quad 
X^2_1 =\partial_{x^2_1}-\frac{x^1_1}{2}\partial_{x^3_1}, \quad X^3_1 =\partial_{x^3_1}.
\end{split}
\end{equation*}
In this case, $s=3$. The structural constants are zero except for $c(1,1)^1_{12}=-1$ and $c(1,2)^1_{11}=-1$. Recall that $\left(X^3_1\right)^r=X^3_1$. First, we compute $X^1_1$ and $X^1_2$ in terms of the right-invariant frame:
\begin{equation*}
\begin{split}
X^1_1 \overset{\eqref{eq_right_vector_fields}}&{=} \left(X^1_1\right)^r +\sum_{i_1=1}^2 \sum_{\alpha_1=1}^2 c(i_1,1)^1_{\alpha_1,1} x^{i_1}_{\alpha_1} \left(X^{1+i_1}_1\right)^r + \frac{1}{2} \sum_{\alpha_1=1}^2 \sum_{\alpha_2=1}^2 c(1,1)^1_{\alpha_1,1} c(1,2)^1_{\alpha_2,1} x^1_{\alpha_1 }x^1_{\alpha_2} X^3_1\\
& = \left(X^1_1\right)^r + c(1,1)^1_{21} x^1_2 \left(X^2_1\right)^r + c(2,1)^1_{11} x^2_1 X^3_1 + \frac{1}{2} c(1,1)^1_{21} c(1,2)^1_{11} x^1_2 x^1_1 X^3_1\\
& = \left(X^1_1\right)^r + x^1_2 \left(X^2_1\right)^r + x^2_1 X^3_1 - \frac{1}{2} x^1_1 x^1_2 X^3_1\\
& = \left(X^1_1\right)^r + x^1_2 \left(X^2_1\right)^r +\left(x^2_1-\frac{x^1_1 x^1_2}{2}\right)X^3_1,\\
X^1_2 \overset{\eqref{eq_right_vector_fields}}&{=} \left(X^1_2\right)^r +\sum_{i_1=1}^2 \sum_{\alpha_1=1}^2 c(i_1,1)^1_{\alpha_1,2} x^{i_1}_{\alpha_1} \left(X^{1+i_1}_1\right)^r + \frac{1}{2} \sum_{\alpha_1=1}^2 \sum_{\alpha_2=1}^2 c(1,1)^1_{\alpha_1,2} c(1,2)^1_{\alpha_2,1} x^1_{\alpha_1 }x^1_{\alpha_2} X^3_1\\
& = \left(X^1_2\right)^r + c(1,1)^1_{12} x^1_1 \left(X^2_1\right)^r + \frac{1}{2} c(1,1)^1_{12} c(1,2)^1_{11} \left(x^1_1\right)^2 X^3_1\\
& = \left(X^1_2\right)^r - x^1_1 \left(X^2_1\right)^r + \frac{\left(x^1_1\right)^2}{2} X^3_1.
\end{split}
\end{equation*}
Finally,
\begin{equation*}
\begin{split}
X^2_1 \overset{\eqref{eq_right_vector_fields}}{=} \left(X^2_1\right)^r + \sum_{\alpha_1=1}^2 c(1,2)^1_{\alpha_1,1} x^1_{\alpha_1} X^3_1 = \left(X^2_1\right)^r + c(1,2)^1_{11} x^1_1 X^3_1 = \left(X^2_1\right)^r - x^1_1 X^3_1.
\end{split}
\end{equation*}
\end{example}

\section{Contact vector fields} \label{sec_contact}
As stressed in the introduction, \emph{contact vector fields} constitute a natural class in the setting of stratified Lie groups. A smooth vector field $\bb \in \mathbb{X}(\mathbb{G})$ is a contact vector field if
\begin{equation} \label{eq_commutator_contact}
[\bb,V_1]_p \subseteq \hhh_p \qquad \text{for every $p\in\mathbb G$.}
\end{equation}
The interest in this class lies in the fact that contact vector fields are precisely those vector fields whose flows both preserve the horizontal distribution and are (locally) Lipschitz continuous in the metric space $(\mathbb G,d)$. Since we are interested in non-smooth vector fields, we notice that \eqref{eq_commutator_contact} is well-posed (in the weak sense) under mild regularity assumptions on $\bb$. For instance, when the components of $\bb$ have horizontal Sobolev regularity, then $[\bb,Z]$ is a well-defined vector field in $\mathbb X(\mathbb{G})$ for every $Z\in V_1$.
Accordingly, we say that $\bb\in \mathbb X(\mathbb{G})$ is a \emph{contact vector field with horizontal Sobolev regularity} if, for some $\theta\in [1,\infty]$,
\begin{equation}\label{def_contact_SObolev_reg}
    \bb=\sum_{i=1}^s\sum_{\alpha=1}^{n_i}b^i_\alpha X^i_\alpha,\qquad b^i_\alpha\in W^{1,\theta}_{\hhh,\mathrm{loc}}(\mathbb G),\qquad [\bb,V_1]_p\subseteq\hhh_p\,\text{ for a.e.~$p\in\mathbb G$.}
\end{equation}
The contact property in \eqref{def_contact_SObolev_reg} can be written as a family of compatibility conditions. Indeed, if $\beta=1,\ldots,n_1$,
\begin{equation*}
    \begin{split}
        \left[\bb,X^1_\beta\right]&=\sum_{k=1}^s\sum_{\alpha=1}^{n_k}b^k_\alpha\left[X^k_\alpha,X^1_\beta\right]-\sum_{k=1}^s\sum_{\gamma=1}^{n_k}\left(X^1_\beta b^k_\gamma\right)X^k_\gamma\\
        \overset{\eqref{eq_structural_constants}}&{=}\sum_{k=1}^{s-1}\sum_{\gamma=1}^{n_{k+1}}\sum_{\alpha=1}^{n_k}c(k,1)_{\alpha\beta}^\gamma \, b^k_\alpha\,X^{k+1}_\gamma-\sum_{k=2}^s\sum_{\gamma=1}^{n_k}\left(X^1_\beta b^k_\gamma\right)X^k_\gamma-\sum_{\gamma=1}^{n_1}\left(X^1_\beta b^1_\gamma\right)X^1_\gamma\\
        &=\sum_{k=2}^{s}\sum_{\gamma=1}^{n_{k}}\left(\sum_{\alpha=1}^{n_{k-1}}c(k-1,1)_{\alpha\beta}^\gamma \, b^{k-1}_\alpha-X^1_\beta b^k_\gamma\right)X^k_\gamma-\sum_{\gamma=1}^{n_1}\left(X^1_\beta b^1_\gamma\right)X^1_\gamma.
    \end{split}
\end{equation*}
In particular, \eqref{def_contact_SObolev_reg} is equivalent to
\begin{equation} \label{eq_contact_components_first}
X^1_\beta b^k_\gamma = \sum_{\alpha=1}^{n_{k-1}} c(k-1,1)^\gamma_{\alpha \beta} \, b^{k-1}_\alpha\qquad \text{for every $\beta=1,\dots,n_1$, $k=2,\dots,s$, $\gamma=1,\dots,n_k$.}
\end{equation}
By \eqref{eq_contact_components_first}, the regularity of $\bb$ improves in the following stratified way.
\begin{proposition}\label{rem_iterated_contact_components}
    Let $\bb$ satisfy \eqref{def_contact_SObolev_reg}. Then, for every $i=1,\ldots,s$,
    \begin{equation}\label{soboleveqreg}
        b^i_\alpha\in W^{i,\theta}_{\hhh,\mathrm{loc}}(\mathbb G)\qquad\text{for every $\alpha=1,\ldots,n_i$.}
    \end{equation}
\end{proposition}
\begin{proof} 
We argue by induction on $i$. The case $i=1$ is part of the assumption \eqref{def_contact_SObolev_reg}. Assume $i>1$ and that \eqref{soboleveqreg} holds for every $\alpha=1,\dots,n_{i-1}$. Then, for any $\beta=1,\ldots,n_1$ and $\gamma=1,\dots,n_i$,
\begin{equation*}
X^1_\beta b^i_\gamma \overset{\eqref{eq_contact_components_first}}{=} \sum_{\alpha=1}^{n_{i-1}} c(i-1,1)^\gamma_{\alpha \beta} b^{i-1}_\alpha \in W^{i-1,\theta}_{\hhh,\mathrm{loc}}(\mathbb{G}).
\end{equation*}
In particular, $b^i_\gamma \in W^{i,\theta}_{\hhh,\mathrm{loc}}(\mathbb{G})$, and the thesis follows.
\end{proof}
The improved regularity provided by \Cref{rem_iterated_contact_components} allows to enlarge the class of compatibility conditions.
\begin{proposition} \label{prop_contact_components}
  Let $\bb$ satisfy \eqref{def_contact_SObolev_reg}. Then
\begin{equation} \label{eq_contact_components}
X^j_\beta b^k_\gamma = \sum_{\alpha=1}^{n_{k-j}} c(k-j,j)^\gamma_{\alpha \beta} b^{k-j}_\alpha
\end{equation}
for every $j=1,\dots,s-1$, $\beta=1,\dots,n_j$, $k=j+1,\dots,s$ and $\gamma=1,\dots,n_k$.
\end{proposition}
\begin{proof}
We argue by induction on $j$. The case $j=1$ is true by \eqref{eq_contact_components_first}. Assume $j>1$, and assume that \eqref{eq_contact_components} holds for $j-1$ and for any $\beta=1,\dots,n_{j-1}$, $k=j,\dots,s$ and $\gamma=1,\dots,n_k$.
Fix $\delta=1,\dots,n_1$, $
\beta=1,\dots,n_{j-1}$, $
k=j+1,\dots,s$, 
$
\gamma=1,\dots,n_k$.
By induction,
\begin{equation*}
X^{j-1}_\beta b^k_\gamma
=
\sum_{\alpha=1}^{n_{k-j+1}}
c(k-j+1,j-1)^\gamma_{\alpha\beta} \,
b^{k-j+1}_\alpha.
\end{equation*}
Therefore,
\begin{equation*}
X^1_\delta\left(X^{j-1}_\beta b^k_\gamma\right)
=
\sum_{\alpha=1}^{n_{k-j+1}}
c(k-j+1,j-1)^\gamma_{\alpha\beta} \,
X^1_\delta b^{k-j+1}_\alpha
\overset{\eqref{eq_contact_components_first}}{=}
\sum_{\alpha=1}^{n_{k-j+1}}
\sum_{\mu=1}^{n_{k-j}}
c(k-j+1,j-1)^\gamma_{\alpha\beta} \,
c(k-j,1)^\alpha_{\mu\delta} \,
b^{k-j}_\mu.
\end{equation*}
On the other hand, by \eqref{eq_contact_components_first},
\begin{equation*}
X^1_\delta b^k_\gamma
=
\sum_{\alpha=1}^{n_{k-1}}
c(k-1,1)^\gamma_{\alpha\delta} \,
b^{k-1}_\alpha.
\end{equation*}
Hence, by induction,
\begin{equation*}
X^{j-1}_\beta\left(X^1_\delta b^k_\gamma\right)
=
\sum_{\alpha=1}^{n_{k-1}}
c(k-1,1)^\gamma_{\alpha\delta} \,
X^{j-1}_\beta b^{k-1}_\alpha
=
\sum_{\alpha=1}^{n_{k-1}}
\sum_{\mu=1}^{n_{k-j}}
c(k-1,1)^\gamma_{\alpha\delta} \,
c(k-j,j-1)^\alpha_{\mu\beta} \,
b^{k-j}_\mu.
\end{equation*}
Combining the above equations,
\begin{equation*}
\left[X^1_\delta,X^{j-1}_\beta\right]b^k_\gamma
=
\sum_{\mu=1}^{n_{k-j}}
\left(
\sum_{\alpha=1}^{n_{k-j+1}}
c(k-j,1)^\alpha_{\mu\delta} \,
c(k-j+1,j-1)^\gamma_{\alpha\beta} \,
-
\sum_{\alpha=1}^{n_{k-1}}
c(k-j,j-1)^\alpha_{\mu\beta} \,
c(k-1,1)^\gamma_{\alpha\delta}
\right)b^{k-j}_\mu.
\end{equation*}
By the Jacobi identity \eqref{eq_Jacobi}, if $\mu=1,\ldots,n_{k-j}$,
\begin{equation*}
\left[\left[X^{k-j}_\mu,X^1_\delta\right],X^{j-1}_\beta\right]
-
\left[\left[X^{k-j}_\mu,X^{j-1}_\beta\right],X^1_\delta\right]
=
\left[X^{k-j}_\mu,\left[X^1_\delta,X^{j-1}_\beta\right]\right].
\end{equation*}
Applying \eqref{eq_structural_constants} to the above three terms, we obtain
\begin{align*}
\left[\left[X^{k-j}_\mu,X^1_\delta\right],X^{j-1}_\beta\right]
&=
\sum_{\alpha=1}^{n_{k-j+1}}
\sum_{\eta=1}^{n_k}
c(k-j,1)^\alpha_{\mu\delta} \,
c(k-j+1,j-1)^\eta_{\alpha\beta} \,
X^k_\eta,
\\
\left[\left[X^{k-j}_\mu,X^{j-1}_\beta\right],X^1_\delta\right]
&=
\sum_{\alpha=1}^{n_{k-1}}
\sum_{\eta=1}^{n_k}
c(k-j,j-1)^\alpha_{\mu\beta} \,
c(k-1,1)^\eta_{\alpha\delta} \,
X^k_\eta,
\\
\left[X^{k-j}_\mu,\left[X^1_\delta,X^{j-1}_\beta\right]\right]
&=
\sum_{\lambda=1}^{n_j}
\sum_{\eta=1}^{n_k}
c(1,j-1)^\lambda_{\delta\beta} \,
c(k-j,j)^\eta_{\mu\lambda} \,
X^k_\eta.
\end{align*}
Since $\left\{X^k_1,\dots,X^k_{n_k}\right\}$ is a basis for $V_k$, we conclude that 
\begin{equation*}
\sum_{\alpha=1}^{n_{k-j+1}}
c(k-j,1)^\alpha_{\mu\delta} \,
c(k-j+1,j-1)^\gamma_{\alpha\beta}
-
\sum_{\alpha=1}^{n_{k-1}}
c(k-j,j-1)^\alpha_{\mu\beta} \,
c(k-1,1)^\gamma_{\alpha\delta}
=
\sum_{\lambda=1}^{n_j}
c(1,j-1)^\lambda_{\delta\beta} \,
c(k-j,j)^\gamma_{\mu\lambda}.
\end{equation*}
In particular,
\begin{equation}\label{aux_improved_uno}
    \left[X^1_\delta,X^{j-1}_\beta\right]b^k_\gamma
=
\sum_{\lambda=1}^{n_j}
\sum_{\mu=1}^{n_{k-j}}
c(1,j-1)^\lambda_{\delta\beta} \,
c(k-j,j)^\gamma_{\mu\lambda} \,
b^{k-j}_\mu.
\end{equation}
On the other hand,
\begin{equation}\label{aux_improved_due}
\left[X^1_\delta,X^{j-1}_\beta\right]b^k_\gamma\overset{\eqref{eq_structural_constants}}{=}
\sum_{\lambda=1}^{n_j}
c(1,j-1)^\lambda_{\delta\beta} \, X^j_\lambda b^k_\gamma.
\end{equation}
Therefore, comparing \eqref{aux_improved_uno} and \eqref{aux_improved_due} and setting 
\begin{equation*}
    X=\sum_{\lambda=1}^{n_j}
\left(
X^j_\lambda b^k_\gamma
-
\sum_{\mu=1}^{n_{k-j}}
c(k-j,j)^\gamma_{\mu\lambda} \,
b^{k-j}_\mu
\right)X^j_\lambda,
\end{equation*}
we have 
\begin{equation*}
    \left\langle\left[X^1_\delta,X^{j-1}_\beta\right],X\right\rangle=\sum_{\lambda=1}^{n_j}
c(1,j-1)^\lambda_{\delta\beta}
\left(
X^j_\lambda b^k_\gamma
-
\sum_{\mu=1}^{n_{k-j}}
c(k-j,j)^\gamma_{\mu\lambda} \,
b^{k-j}_\mu
\right)=0
\end{equation*}
for every $\delta=1,\ldots,n_1$ and $\beta=1,\ldots,n_{j-1}$. This fact and \eqref{defcarnotgroup} grant that $X=0$, whence \eqref{eq_contact_components} follows.
\end{proof}
The forthcoming commutator estimates will require the following careful iterations of \eqref{eq_contact_components}.
\begin{proposition}\label{peodneicjecneficnecijn}
 Let $\bb$ satisfy \eqref{def_contact_SObolev_reg}. Let  $k \in \mathbb{N}_+$, $i=2,\dots,s$, $\abs*{I_k} < i$, $A_k \in \mathcal A_k(I_k)$, $\beta_k=1,\dots,n_i$. Then 
\begin{equation} \label{eq_iterated_contact_components}
X^{I_k}_{A_k} b^i_{\beta_k} = \sum_{\mathcal B_{k-1}(i-\abs*{I_k},I_{k-1})} \sum_{\alpha=1}^{n_{i-\abs*{I_k}}} c(i,I_k)_{B_{k-1} A_k}^{B_k} b^{i-\abs*{I_k}}_\alpha
\end{equation}
\end{proposition}
\begin{proof}
We argue by induction on $k$. We know that
\begin{equation*}
X^{i_1}_{\alpha_1} b^i_{\beta_1} \overset{\eqref{eq_contact_components}}{=} \sum_{\alpha=1}^{n_{i-i_1}} c(i-i_1,i_1)^{\beta_1}_{\alpha \alpha_1} b^{i-i_1}_\alpha,
\end{equation*}
which is the thesis for $k=1$. Then, assume $k>1$ and that the thesis holds for $k-1$, so that
\begin{equation*}
\begin{split}
    X^{i_2}_{\alpha_2} \dots X^{i_k}_{\alpha_k} b^i_{\beta_k} & = \sum_{\substack{(\beta_2,\dots,\beta_{k-1}) \in \mathbb{N}^{k-2}\\1 \leq \beta_j \leq n_{i-i_{j+1}-\dots-i_k}}} \sum_{\beta_1=1}^{n_{i-i_2-\dots-i_k}} c(i-i_k-\dots-i_2,i_2)^{\beta_2}_{\beta_1 \alpha_2} \dots c(i-i_k,i_k)^{\beta_k}_{\beta_{k-1} \alpha_k} b^{i-i_2-\dots-i_k}_{\beta_1}\\
    & = \sum_{\mathcal B_{k-1}(i-\abs*{I_k},I_{k-1})} c(i-i_k-\dots-i_2,i_2)^{\beta_2}_{\beta_1 \alpha_2} \dots c(i-i_k,i_k)^{\beta_k}_{\beta_{k-1} \alpha_k} b^{i-i_2-\dots-i_k}_{\beta_1}.
\end{split}
\end{equation*}
Differentiating along $X^{i_1}_{\alpha_1}$, we finally obtain
\begin{equation*}
\begin{split}
    X^{I_k}_{A_k} b^i_{\beta_k} & = \sum_{\mathcal B_{k-1}(i-\abs*{I_k},I_{k-1})} c(i-i_k-\dots-i_2,i_2)^{\beta_1}_{\beta_1 \alpha_2} \dots c(i-i_k,i_k)^{\beta_k}_{\beta_{k-1} \alpha_k} X^{i_1}_{\alpha_1} b^{i-i_2-\dots-i_k}_{\beta_1}\\
    \overset{\eqref{eq_contact_components}}&{=} \sum_{\mathcal B_{k-1}(i-\abs*{I_k},I_{k-1})} \sum_{\alpha=1}^{n_{i-\abs*{{I_k}}}} c(i-\abs*{I_k},i_1)^{\beta_1}_{\alpha \alpha_1} c(i-i_k-\dots-i_2,i_2)^{\beta_1}_{\beta_1 \alpha_2} \dots c(i-i_k,i_k)^{\beta_k}_{\beta_{k-1} \alpha_k} b^{i-\abs*{I_k}}_\alpha\\
    & = \sum_{\mathcal B_{k-1}(i-\abs*{I_k},I_{k-1})} \sum_{\alpha=1}^{n_{i-\abs*{I_k}}} c(i,I_k)_{B_{k-1} A_k}^{B_k} b^{i-\abs*{I_k}}_\alpha,
\end{split}
\end{equation*}
which is the desired formula.
\end{proof}
\begin{remark}
    We proved that a vector field $\bb$ as in \eqref{def_contact_SObolev_reg} satisfies \eqref{eq_iterated_contact_components}. On the other hand, if    
    \begin{equation*}
    \bb=\sum_{i=1}^s\sum_{\alpha=1}^{n_i}b^i_\alpha X^i_\alpha,\qquad b^i_\alpha\in L^{\theta}_{\mathrm{loc}}(\mathbb G)
\end{equation*}
and \eqref{eq_iterated_contact_components} holds, clearly $\bb$ satisfies \eqref{def_contact_SObolev_reg}. Therefore, \eqref{def_contact_SObolev_reg}, \eqref{eq_contact_components_first}, \eqref{eq_contact_components} and \eqref{eq_iterated_contact_components} are actually equivalent.
\end{remark}
\subsection{Further remarks on contact vector fields}
\Cref{prop_contact_components} may suggest that the components along the last layer uniquely determine a contact vector field. Actually, this is not true, as the next example shows.
\begin{example}
Denote by $\hh^n$ the $n$-th Heisenberg group and consider the direct product of groups $\hh^n \times \rr$. Using the notation of \Cref{ex_Heisenberg} and denoting by $\partial_x$ the canonical vector field on $\rr$, we have that $\hh^n \times \rr$ is a stratified Lie group of step $2$, and its stratification is
\begin{equation*}
V_1=\spann\left\{X^1_1,\dots,X^1_{2n},\partial_x\right\}, \qquad V_2=\spann\left\{X^2_1\right\}.
\end{equation*}
In this setting, a contact vector field $\bb$ decomposes as the sum of a contact vector field on $\hh^n$ and a vector field on $\rr$. In particular, the component of $\bb$ along $\partial_x$ is arbitrary.
\end{example}
Anyway, the orthogonal projection onto the \emph{center} of $\mathfrak{g}$ uniquely determines a contact vector field. Recall that the center $\mathcal{Z}(\mathfrak{g})$ of $\mathfrak{g}$ is the subspace of vector fields $Z \in \mathfrak{g}$ such that $[Z,X]=0$ for every $X \in \mathfrak{g}$.
\begin{proposition}
Let $\bb,\bar{\bb} \in \mathbb{X}(\mathbb{G})$ satisfy \eqref{def_contact_SObolev_reg}. Assume that $\left\langle \bb,\mathcal{Z}(\mathfrak{g})\right\rangle= \left\langle \bar\bb,\mathcal{Z}(\mathfrak{g})\right\rangle$ a.e.~on $\mathbb G$. 
Then $\bb=\bar \bb$.
\end{proposition}
\begin{proof}
First, notice that $$\mathcal{Z}(\mathfrak{g})=\mathcal{Z}(\mathfrak{g})\cap V_1\oplus\cdots\oplus \mathcal{Z}(\mathfrak{g})\cap V_s.$$ It is clear that $\mathcal{Z}(\mathfrak{g})\cap V_1\oplus\cdots\oplus \mathcal{Z}(\mathfrak{g})\cap V_s \subseteq\mathcal{Z}(\mathfrak{g})$. Conversely, let $X\in \mathcal{Z}(\mathfrak{g})$, and decompose it uniquely along $V_1,\ldots,V_s$ as $X=X_1+\cdots+ X_s$. Fix $i=1,\ldots,s$ and $Y\in V_i$. Then 
\begin{equation*}
    0=[X,Y]=\sum_{j=1}^s[X_j,Y].
\end{equation*}
Noticing that $\left\langle[X_j,Y],[X_k,Y]\right\rangle=0$ for every $j,k=1,\ldots,s$ and $j\neq k$, we get $[X_j,Y]=0$ for every $j=1,\ldots,s$, whence $X\in \mathcal{Z}(\mathfrak{g})\cap V_1\oplus\cdots\oplus \mathcal{Z}(\mathfrak{g})\cap V_s$. 
Therefore, up to changing and rearranging the basis, we can assume that $\left\{X^i_1,\dots,X^i_{m_i}\right\}$ is a basis for $\mathcal{Z}(\mathfrak{g}) \cap V_i$ for every $i=1,\dots,s$ and a suitable $0 \leq m_i \leq n_i$. By linearity, we can suppose $\bar \bb=0$. Hence, if $\bb= \sum_{i=1}^s \sum_{\alpha=1}^{n_i} b^i_\alpha X^i_\alpha$, we have $b^i_\alpha=0$ for every $i=1,\dots,s$ and $\alpha=1,\dots,m_i$. It suffices to prove that $b^k \coloneqq \left(b^k_1,\dots,b^k_{n_k}\right)=0$ for every $k=1,\dots,s$. We prove this fact by reverse induction on $k$. Since $V_s \subseteq \mathcal{Z}(\mathfrak{g})$ by definition of stratification, we have $b^s=0$ by assumption. Hence, assume $k<s$ and $b^{k+1}=0$. Fix $p \in \mathbb{G}$ and define the vector field $B=\sum_{\alpha=1}^{n_k} b^k_\alpha(p) X^k_\alpha \in \mathfrak{g}$. Then, we deduce
\begin{equation*}
\left[B,X^1_\beta \right] \overset{\eqref{eq_structural_constants}}{=} \sum_{\gamma=1}^{n_{k+1}} \sum_{\alpha=1}^{n_k} c(k,1)^\gamma_{\alpha \beta} \, b^k_\alpha(p) X^{k+1}_\gamma \overset{\eqref{eq_contact_components}}{=} \sum_{\gamma=1}^{n_{k+1}} \left(X^1_\beta b^{k+1}_\gamma\right) X^{k+1}_\gamma =0
\end{equation*}
for every $\beta=1,\dots,n_1$, so that $B \in \mathcal{Z}(\mathfrak{g})$. On the other hand, we know that $B=\sum_{\alpha=m_i+1}^{n_i} b^k_\alpha(p) X^k_\alpha$ with $X^k_\alpha \notin \mathcal{Z}(\mathfrak{g})$. Hence, the only possibility is $B=0$. This is the thesis, since $p \in \mathbb{G}$ was arbitrary.
\end{proof}
On the other hand, it is not true that all the central components are free. Although this is well-known from the theory of Tanaka prolongation (see \cite{MR266258}), we include the following example for the sake of completeness.
\begin{example}
Consider $\mathbb{F}_{32}$, i.e., the $6$-dimensional stratified Lie group of step $2$ such that
\begin{equation*}
V_1=\spann\{X_1,X_2,X_3\}, \qquad V_2=\spann\{Y_{12},Y_{13},Y_{23}\},\qquad [X_i,X_j]=Y_{ij}\text{ for every $i,j=1,2,3$, $i \neq j$.}
\end{equation*}
Let $\bb=f_1 X_1+f_2 X_2+f_3 X_3+f_{12}Y_{12}+f_{13}Y_{13}+f_{23}Y_{23} \in \mathbb{X}(\mathbb{F}_{32})$ be a smooth contact vector field. A direct computation shows that, for every $i=1,2,3$,
\begin{equation*}
[X_i,\bb]=(X_i f_1) X_1+(X_i f_2) X_2+(X_i f_3) X_3+(X_i f_{12}) Y_{12}+(X_i f_{13}) Y_{13}+(X_i f_{23}) Y_{23}+f_1 Y_{i1}+f_2 Y_{i2}+f_3 Y_{i3},
\end{equation*}
where we use the convention $Y_{ij}=0$ whenever $i=j$. By \eqref{eq_commutator_contact}, the components of $[X_i,\bb]$ along $Y_{12},Y_{13},Y_{23}$ must be zero. In particular,
\begin{equation*}
X_1 f_{23}=X_2 f_{13}=X_3 f_{12}=0,
\end{equation*}
which means that $f_{12},f_{13},f_{23}$ are not arbitrary.
\end{example}
\section{Difference quotients of horizontal Sobolev functions} \label{sec_difference_quotients}
It is well known that Euclidean Sobolev functions are characterized by the convergence (in the suitable weak sense) of their difference quotients. In this section, we prove the analogous result for horizontal Sobolev functions in Carnot groups. This will be crucial for the regularization property of the next section. We stress that we deal with difference quotients computed along intrinsic dilations, involving certain polynomials which will turn out to be Taylor polynomials in the sense of \cite{MR0657581,BonLanUgu}. Throughout this section, $\Omega \subseteq \mathbb{G}$ is a fixed open set. Moreover, whenever $k \in \mathbb{N}$, $f \in C^k(\Omega)$ and $Z \in \mathfrak{g}$, we adopt the notation
\begin{equation*}
Z^k f = \underbrace{Z \dots Z}_{k} f.
\end{equation*}
First, we recall the following two results. See \cite[Lemma 2.12.1]{MR746308} for the proof of the first one.
\begin{lemma}
Let $k \in \mathbb{N}$. Let $f \in C^k(\Omega)$. Fix $p \in \Omega$, $t>0$ and $Z \in \mathfrak{g}$. Assume that $p \cdot \exp(tZ) \in \Omega$. Then
\begin{equation} \label{eq_iterated_derivative}
\frac{d^k}{dt^k} f(p \cdot \exp(tZ)) = Z^k f (p \cdot \exp(tZ)).
\end{equation}
In particular,
\begin{equation} \label{eq_iterated_derivative_zero}
\frac{d^k}{dt^k}\bigg\rvert_{t=0} f(p \cdot \exp(tZ)) = Z^k f(p).
\end{equation}
\end{lemma}
As an immediate consequence of \eqref{eq_iterated_derivative_zero}, observe that, by the Baker-Campbell-Hausdorff formula (cf.~\cite{MR746308}), for every $i,j=1,\dots,s$, $\alpha=1,\dots,n_i$ and $\beta=1,\dots,n_j$,
\begin{equation} \label{eq_derivatives_coordinates}
X^i_\alpha x^i_\alpha = 1, \qquad X^i_\alpha x^i_\beta = 0 \quad \text{if $\alpha \neq \beta$,} \qquad X^i_\alpha x^j_\beta = 0 \quad \text{if $j<i$.}
\end{equation}
\begin{proposition}
Let $k \in \mathbb{N}_+$ and $f \in C^{k+1}(\Omega)$. Fix $p \in \Omega$ and $W=\displaystyle \sum_{i=1}^s \sum_{\alpha=1}^{n_i} w^i_\alpha X^i_\alpha \in \mathfrak{g}$. Then
\begin{equation} \label{eq_Taylor_coefficients}
\frac{d^k}{d\tau^k}\bigg\rvert_{\tau=0} f(p \cdot \exp(\delta_\tau W)) = k! \sum_{j=1}^k \frac{1}{j!} \sum_{\substack{i_1,\dots,i_j=1 \\ i_1+\dots+i_j=k}}^s \sum_{\alpha_1=1}^{n_{i_1}} \dots \sum_{\alpha_j=1}^{n_{i_j}} w^{i_1}_{\alpha_1} \dots w^{i_j}_{\alpha_j} X^{i_1}_{\alpha_1} \dots X^{i_j}_{\alpha_j}f(p).
\end{equation}
\end{proposition}
\begin{proof}
Since $\Omega$ is open, there exists $\tau>0$ such that $p \cdot \exp(t\delta_\tau W) \in \Omega$ for every $t \in [0,1]$. Set $Z=\delta_\tau W$. Then, thanks to \eqref{eq_iterated_derivative}, \eqref{eq_iterated_derivative_zero} and Taylor's formula with Lagrange remainder,
\begin{equation*}
f(p \cdot \exp(tZ))=\sum_{j=0}^k \frac{Z^j f(p)}{j!} t^j + \frac{Z^{k+1}f(p \cdot \exp(\sigma Z))}{(k+1)!}t^{k+1}
\end{equation*}
for a suitable $\sigma \in [0,t]$. In particular, evaluating at $t=1$,
\begin{equation*}
f(p \cdot \exp(Z))=\sum_{j=0}^k \frac{Z^j f(p)}{j!} + \frac{Z^{k+1}f(p \cdot \exp(\sigma Z))}{(k+1)!}.
\end{equation*}
Hence, noticing that, for every $j \geq 1$,
\begin{equation*}
Z^j=(\delta_\tau W)^j=\left(\sum_{i=1}^s \sum_{\alpha=1}^{n_i} \tau^i w^i_\alpha X^i_\alpha\right)^j=\sum_{i_1,\dots,i_j=1}^s \sum_{\alpha_1=1}^{n_{i_1}} \dots \sum_{\alpha_j=1}^{n_{i_j}} \tau^{i_1+\dots+i_j} w^{i_1}_{\alpha_1} \dots w^{i_j}_{\alpha_j} X^{i_1}_{\alpha_1} \dots X^{i_j}_{\alpha_j},
\end{equation*}
we obtain
\begin{equation*}
\begin{split}
    f(p & \cdot \exp(\delta_\tau W)) = f(p) + \sum_{j=1}^k \frac{1}{j!} \sum_{i_1,\dots,i_j=1}^s \sum_{\alpha_1=1}^{n_{i_1}} \dots \sum_{\alpha_j=1}^{n_{i_j}} \tau^{i_1+\dots+i_j} w^{i_1}_{\alpha_1} \dots w^{i_j}_{\alpha_j} X^{i_1}_{\alpha_1} \dots X^{i_j}_{\alpha_j}f(p) \\ & + \underbrace{\frac{1}{(k+1)!} \sum_{i_1,\dots,i_{k+1}=1}^s \sum_{\alpha_1=1}^{n_{i_1}} \dots \sum_{\alpha_{k+1}=1}^{n_{i_{k+1}}} \tau^{i_1+\dots+i_{k+1}} w^{i_1}_{\alpha_1} \dots w^{i_{k+1}}_{\alpha_{k+1}} X^{i_1}_{\alpha_1} \dots X^{i_{k+1}}_{\alpha_{k+1}}f(p \cdot \exp(\sigma \delta_\tau W))}_{\eqqcolon g(\tau)}.
\end{split}
\end{equation*}
Therefore, since $g(\tau)=o\left(\tau^k\right)$ as $\tau \searrow 0$, the thesis follows by the uniqueness of Taylor's polynomial.
\end{proof}
Formula \eqref{eq_Taylor_coefficients} motivates the following definition. Given $p \in \Omega$, $q \in \mathbb{G}$ and $f \in C^\infty(\Omega)$, set
\begin{equation} \label{eq_higher_order_derivatives}
T^k_p f(q) \coloneqq \sum_{j=1}^k \frac{1}{j!} \sum_{\substack{i_1,\dots,i_j=1 \\ i_1+\dots+i_j=k}}^s \sum_{\alpha_1=1}^{n_{i_1}} \dots \sum_{\alpha_j=1}^{n_{i_j}} q^{i_1}_{\alpha_1} \dots q^{i_j}_{\alpha_j} X^{i_1}_{\alpha_1} \dots X^{i_j}_{\alpha_j}f(p)
\end{equation}
and
\begin{equation} \label{eq_Taylor_polynomial}
P^\ell_p f(q) \coloneqq f(p)+\sum_{k=1}^\ell T^k_p f(q).
\end{equation}
We say that $P^\ell_p f$ is the \emph{Taylor polynomial of $f$ of order $\ell$ centered at $p$}. Combining \cite[Corollary 20.2.10]{BonLanUgu} and \cite[Proposition 20.3.15]{BonLanUgu},
\begin{equation} \label{eq_equivalence_Taylor}
X^{I_k}_{A_k} P^\ell_p f(0) = X^{I_k}_{A_k} f(p) \qquad \text{for every $k \leq \ell$, $\abs*{I_k} \leq \ell$ and $A_k \in \mathcal A_k(I_k)$.}
\end{equation}
Equivalently, $P^\ell_p f$ is the only polynomial with \emph{homogeneous degree} (cf.~\cite{MR0657581,BonLanUgu}) at most $\ell$ such that \eqref{eq_equivalence_Taylor} holds. We can now proceed with the asymptotic analysis of the difference quotients and prove the main result of the section. For any $\eps>0$ and $w \in \mathbb{G}$, set
\begin{equation} \label{eq_difference_quotients_notation}
    R^w_{\ell,\eps} f(p) \coloneqq \frac{f\left(p\cdot\delta_\eps(w)\right)-P^\ell_p f\left(\delta_\eps(w)\right)}{\eps^\ell}.
\end{equation}
\begin{theorem} \label{thm_difference_quotients}
Let $\ell \geq 1$ and $\theta \in [1,\infty)$. Let $f \in W_\hhh^{\ell,\theta}(\Omega)$. Let $A\subseteq\Om$ be open with $\mathrm{dist}(A,\partial\Om)>0$. Then
\begin{equation}\label{eq:priori10_lavorocarnot}
\lim_{\eps \searrow 0}^{L^\theta(A)}  R^w_{\ell,\eps} f(p) = 0\qquad\text{for every $w \in \mathbb G$.}
\end{equation}
In addition, there exist $C_1=C_1(\mathbb G)>0$ and $C_2=C_2(\mathbb G,\ell)>0$ such that, if $\varepsilon>0$ and $w\in\mathbb G$ satisfy 
\begin{equation}\label{condepsw_lavorocarnot}
C_1\,\varepsilon \,d(w,0)<\,\mathrm{dist}(A,\partial\Om),
\end{equation}
then
\begin{equation}\label{eq:priori11_lavorocarnot}
   \left\| R^w_{\ell,\eps}f\right\|_{L^\theta(A)}\leq C_2\, d(w,0)^\ell\sum_{\alpha_1,\ldots,\alpha_\ell=1}^{n_1} \left\|X^1_{\alpha_1}\ldots X^1_{\alpha_\ell}f \right\|_{L^\theta(\Om)}.
\end{equation}
Finally, if $f\in W^{\ell,\infty}_\hhh(\Om)$, then 
\begin{equation}\label{eq:priori11_infinito_lavorocarnot}
   \left\| R^w_{\ell,\eps}f\right\|_{L^\infty(A)}\leq C_2\, d(w,0)^\ell\sum_{\alpha_1,\ldots,\alpha_\ell=1}^{n_1} \left\|X^1_{\alpha_1}\ldots X^1_{\alpha_\ell}f \right\|_{L^\infty(\Om)}.
\end{equation}
\end{theorem}
\begin{proof}
Assume first that $\theta\in[1,\infty)$. By the density of $C^\infty(\Om)\cap W^{\ell,\theta}_\hhh(\Om)$ in $W^{\ell,\theta}_\hhh(\Om)$ (recall \Cref{meyersserrin}), it suffices to prove the thesis when $f\in C^\infty(\Om)\cap W^{\ell,\theta}_\hhh(\Om)$. Fix $p\in A$. 
    Since $\mathrm{dist}(A,\partial\Om)>0$, there exists an open neighborhood of $0$, say $U$, such that the function $F_p:U\to \rr$ given by 
    \begin{equation*}
        F_p(q)\coloneqq f(p\cdot q)-P^\ell_p f(q)
    \end{equation*}
    is well-defined. Notice that
    \begin{equation} \label{derivatesonozeronellidentitaF}
    X^1_{\alpha_1}\ldots X^1_{\alpha_j}F_p(0) \overset{\eqref{eq_equivalence_Taylor}}{=}0\qquad\text{for every $j=0,\ldots,\ell$ and $\alpha_1,\ldots,\alpha_j=1,\ldots, n_1$.}
    \end{equation}
    Moreover, by the definition of $P^\ell_p f$,
    \begin{equation} \label{derivatediFallultimostep}
        X^1_{\alpha_1}\ldots X^1_{\alpha_\ell}F_p(q)=\left(X^1_{\alpha_1}\ldots X^1_{\alpha_\ell}f\right)(p\cdot q)-X^1_{\alpha_1}\ldots X^1_{\alpha_\ell}f(p)\qquad\text{for every $\alpha_1,\ldots \alpha_\ell=1,\ldots, n_1$}.
    \end{equation}
 By \cite[Lemma 1.40]{MR0657581}, there exist constants $N\in\mathbb N_+$ and $C>0$, both depending only on $\mathbb G$, such that, for any $w\in\mathbb G$, the following properties hold:
    \begin{align}
        w&=w_1\cdot w_2 \cdot \ldots \cdot w_N\qquad\text{for some $w_1,\ldots,w_N\in \exp(V_1)$,}\nonumber\\
        d(w_j,0)&\leq C d(w,0)\qquad\text{for every $j=1,\ldots,N$.}\label{boundsupezziorizzontali373737373}
    \end{align}
    Set $C_1=N C$. Fix $w\in\mathbb G$ and $\eps>0$ such that \eqref{condepsw_lavorocarnot} holds. Define $\gamma:[0,1]\to \mathbb G$ by means of the following construction. For every $j=1,\ldots, N$, set
    \begin{align*}
        t_0=0,\qquad t_j\coloneqq \frac{j}{N},\qquad 
        J_j\coloneqq \left[t_{j-1},t_j\right],\qquad W_j\coloneqq \log w_j.
    \end{align*}
    Define $\gamma$ by setting
    \begin{equation*}
        \left.\gamma\right|_{J_j}(\tau)\coloneqq \delta_\eps(w_1) \cdot \ldots \cdot \delta_\eps(w_{j-1})\cdot \delta_{\eps N(\tau-t_{j-1})}(w_j).
    \end{equation*}
    In particular, 
    \begin{equation*}
        \left.\dot\gamma\right|_{J_j}(\tau)=\eps N \left.W_j\right|_{\gamma(\tau)}.
    \end{equation*}
    By construction, $p\cdot \gamma$ is a horizontal, piecewise smooth curve joining $\gamma(0)=p$ and $\gamma(1)=p\cdot\delta_\eps(w)$. Moreover, 
    \begin{equation}\label{stimasugammatauuuuuuu}
        d\left(p\cdot \gamma(\tau),p\right)=d\left(\gamma(\tau),0\right)\leq \eps \sum_{i=1}^N d(w_i,0)\overset{\eqref{boundsupezziorizzontali373737373}}{\leq} NC\eps d(w,0)=C_1 \eps d(w,0).
    \end{equation}
    Therefore, by \eqref{condepsw_lavorocarnot}, $p\cdot \gamma(\tau)\in \Om$ for every $\tau\in [0,1]$. By the fundamental theorem of calculus, writing $\tau_0 \coloneqq 1$,
    \begin{equation*}
        \begin{split}
            F_p\left(\delta_\eps(w)\right)\overset{\eqref{derivatesonozeronellidentitaF}}{=}F_p\left(\gamma(\tau_0)\right)-F_p\left(\gamma(0)\right)=\int_0^{\tau_0}\frac{d}{d\tau_1}\left(F_p(\gamma(\tau_1)\right)\,d\tau_1=\eps N \sum_{i_1=1}^N\int_{J_{i_1}}\left(W_{i_1 }F_p\right)(\gamma(\tau_1))\,d\tau_1.
        \end{split}
    \end{equation*}
   Fix $i_1=1,\ldots, N$ and $\tau_1\in J_{i_1}$. Since $W_{i_1}\in V_1$, \eqref{derivatesonozeronellidentitaF} yields $\left(W_{i_1}F_p\right)(0)=0$. Therefore, arguing as above,
\begin{equation*}
        \begin{split}
            \left(W_{i_1}F_p\right)\left(\gamma(\tau_1)\right)\overset{\eqref{derivatesonozeronellidentitaF}}{=}\left(W_{i_1}F_p\right)\left(\gamma(\tau_1)\right)-\left(W_{i_1}F_p\right)\left(\gamma(0)\right)=\eps N \sum_{i_2=1}^N\int_{J_{i_2}\cap \, [0,\tau_1]}\left(W_{i_2}W_{i_1}F_p\right)(\gamma(\tau_2))\,d\tau_2.
        \end{split}
    \end{equation*}
    Combining the above identities,
    \begin{equation*}
        F_p\left(\delta_\eps(w)\right)=(\eps N)^2\sum_{i_1,i_2=1}^N\int_{J_{i_1}}\int_{J_{i_2}\cap \, [0,\tau_1]}\left(W_{i_2}W_{i_1}F_p\right)(\gamma(\tau_2))\,d\tau_2\,d\tau_1.
    \end{equation*}
    Iterating the above procedure $\ell$ times,
     \begin{equation*}
        F_p\left(\delta_\eps(w)\right)=(\eps N)^\ell\sum_{i_1,\ldots,i_\ell=1}^N\int_{J_{i_1}}\dots\int_{J_{i_\ell}\cap \, [0,\tau_{\ell-1}]}\left(W_{i_\ell}\ldots W_{i_1}F_p\right)(\gamma(\tau_\ell))\,d\tau_\ell\ldots d\tau_1.
    \end{equation*}
    Hence, by \eqref{derivatediFallultimostep} and noticing that $F_p(\delta_\eps(w))=\eps^\ell R^w_{\ell,\eps} f(p)$,
    \begin{equation*}
        R^w_{\ell,\eps}f(p)=N^\ell \sum_{i_1,\ldots,i_\ell=1}^N\int_{J_{i_1}}\dots\int_{J_{i_\ell}\cap \, [0,\tau_{\ell-1}]}\Big[\left(W_{i_\ell}\ldots W_{i_1 }f\right)(p\cdot \gamma(\tau_\ell))-\left(W_{i_\ell}\ldots W_{i_1}f\right)(p)\Big]\,d\tau_\ell\ldots d\tau_1.
    \end{equation*} 
    In particular, Minkowski's integral inequality (cf.~\cite[Corollary B.83]{MR3726909}) implies
    \begin{equation}\label{aftermif56889}
        \left\|R^w_{\ell,\eps}f \right\|_{L^\theta(A)}\leq N^\ell\sum_{i_1,\ldots,i_\ell=1}^N\int_0^1\underbrace{\left(\int_A\Big|\left(W_{i_\ell}\ldots W_{i_1}f\right)(p\cdot \gamma(\tau))-\left(W_{i_\ell}\ldots W_{i_1}f\right)(p)\Big|^\theta\,dp\right)^\frac{1}{\theta}}_{\eqqcolon g_{i_1,\ldots,i_\ell}(\tau)}\,d\tau.
    \end{equation} 
    First, \eqref{stimasugammatauuuuuuu} and \Cref{lem_continuity_translations} imply that 
    \begin{equation}\label{aeconv38586795838587}
        \text{$g_{i_1,\ldots,i_\ell}(\tau)\to 0$ as $\eps\searrow 0$ for a.e.~$\tau\in (0,1)$.}
    \end{equation} 
    For any $i=1,\ldots, N$, set $W_i=\sum_{j=1}^{n_1}w^j_i X^1_j$ and denote by $|w_i|_\mathrm{eu}$ the Euclidean norm of $(w_1^1,\ldots,w_i^{n_1})$. By \cite[Example 5.1.2]{BonLanUgu} and \cite[Proposition 5.1.4]{BonLanUgu}, there exists a constant $\tilde C=\tilde C(\mathbb G)>0$ such that 
\begin{equation}\label{stimaeucondcc}
    |w_i|_\mathrm{eu}\leq\tilde Cd(w_i,0).
\end{equation}
  Then
    \begin{equation*}
        \begin{split}
            \left|W_{i_\ell}\ldots W_{i_1}f(q)\right|
            &\leq \left(n_1\right)^\ell|w_{i_\ell}|_\mathrm{eu}\ldots|w_{i_1}|_\mathrm{eu} \sum_{j_1,\ldots,j_\ell=1}^{n_1} \left|X^1_{j_\ell}\ldots X^1_{j_1}f(q)\right|\\
            \overset{\eqref{stimaeucondcc}}&{\leq} \left(\tilde C n_1\right)^\ell d(w_{i_\ell},0)\ldots d(w_{i_1},0) \sum_{j_1,\ldots,j_\ell=1}^{n_1} \left|X^1_{j_\ell}\ldots X^1_{j_1}f(q)\right|\\
            \overset{\eqref{boundsupezziorizzontali373737373}}&{\leq}\left(C\tilde C n_1\right)^\ell d(w,0)^\ell \sum_{j_1,\ldots,j_\ell=1}^{n_1} \left|X^1_{j_\ell}\ldots X^1_{j_1}f(q)\right|.
        \end{split}
    \end{equation*}
    Therefore, \eqref{eq_Lebesgue_Haar} implies that 
    \begin{equation}\label{sitmfkgotgidjfvof}
        g_{i_1,\ldots,i_\ell}(\tau)\leq 2\left(C\tilde C n_1\right)^\ell d(w,0)^\ell \sum_{j_1,\ldots,j_\ell=1}^{n_1} \left\|X^1_{j_\ell}\ldots X^1_{j_1}f \right\|_{L^\theta(\Om)}.
    \end{equation}
    By \eqref{aftermif56889}, \eqref{aeconv38586795838587}, \eqref{sitmfkgotgidjfvof} and the dominated convergence theorem, \eqref{eq:priori10_lavorocarnot} follows. Moreover, \eqref{eq:priori11_lavorocarnot} follows as well from \eqref{aftermif56889} and \eqref{sitmfkgotgidjfvof}. Eventually, assume that $f\in W^{\ell,\infty}_\hhh(\Om)$. By \cite{MR1631642}, 
    \begin{equation}\label{derivatecontinuefinoaellmenouno}
        X^1_{i_1}\ldots X^1_{i_k}f\in C(\Om)\qquad\text{for every $k=1,\ldots,\ell-1$ and $i_1,\ldots,i_k=1,\ldots,n_1$.}
    \end{equation}
    Combining \eqref{derivatecontinuefinoaellmenouno} with \cite[Proposition 2.4]{MR4581339} and \cite[Proposition 2.5]{MR4581339}, one argues as above to deduce that 
     \begin{equation}\label{finoaunoprimaperr45678}
          R^w_{\ell,\eps}(p)= \frac{N^{\ell-1}}{\eps}\sum_{i_1,\ldots,i_{\ell-1}=1}^N\int_{J_{i_1}}\dots\int_{J_{i_{\ell-1}}\cap[0,\tau_{\ell-2}]}\left(W_{i_{\ell-1}}\ldots W_{i_1 }F_p\right)( \gamma(\tau_{\ell-1}))\,d\tau_{\ell-1}\ldots d\tau_1.
        \end{equation}
        Combining \eqref{derivatesonozeronellidentitaF} and \eqref{stimasugammatauuuuuuu} with \cite[Proposition 3.4]{MR4801826}, \cite[Proposition 2.4]{MR4581339} and \cite[Proposition 2.5]{MR4581339}, 
\begin{equation}\label{stimaconsubdiffimplicito}
    \left|\left(W_{i_{\ell-1}}\ldots W_{i_1 }F_p\right)(\gamma(\tau_{\ell-1})) \right|\leq 2\eps N \left\|W_{i_\ell}\ldots W_{i_1} f \right\|_{L^\infty(\Om)}.
\end{equation}
By \eqref{finoaunoprimaperr45678} and \eqref{stimaconsubdiffimplicito}, \eqref{eq:priori11_infinito_lavorocarnot} follows as in the previous case.
\end{proof}

\section{Well-posedness for transport and flow equations driven by contact velocities} \label{sec_well_posedness}
In this section, we prove that contact vector fields with horizontal Sobolev regularity satisfy the renormalization property. As a consequence, we obtain well-posedness for the transport and flow equations. We provide only the proofs specific to our setting, as the remaining arguments follows as in Euclidean spaces and Heisenberg groups. We refer to \cite{MR1022305} and \cite{dpl1}, respectively, for the corresponding proofs and further details.
\subsection{Transport equation, distributional solutions, and renormalized solutions}\label{sec_prelitranspo}
Fix $\bar\tau\in (0,\infty]$. Given $\alpha,\beta\in [1,\infty]$, we consider (see \cite{MR3726909}) the space $L^\alpha\left(0,\bar\tau;L^\beta(\mathbb{G})\right)$ of Borel functions $u : (0,\bar{\tau}) \times \mathbb{G} \to \rr$ such that 
\begin{equation*}
\norm*{u}_{L^\alpha(0,\bar{\tau};L^\beta(\mathbb{G}))} \coloneqq \begin{cases}
\displaystyle\left(\int_0^{\bar{\tau}} \norm*{u(\cdot,\tau)}_{L^\beta(\mathbb{G})}^\alpha \,d\tau\right)^{\frac{1}{\alpha}} & \text{if $\alpha \in [1,\infty)$}\\
\mathop{\mathrm{ess\,sup}}\limits_{\tau \in (0,\bar{\tau})} \, \norm*{u(\cdot,\tau)}_{L^\beta(\mathbb{G})} & \text{if $\alpha=\infty$}
\end{cases}
\end{equation*}
is finite. Its local version $L^\alpha\left(0,\bar{\tau};L^\beta_{\mathrm{loc}}(\mathbb{G})\right)$ is defined similarly. Moreover, when $\alpha=\beta$, $L^\alpha\left(0,\bar{\tau};L^\alpha(\mathbb{G})\right)$ naturally identifies with $L^\alpha\left((0,\bar{\tau})\times\mathbb{G}\right)$.
The Cauchy problem for the \emph{transport equation} reads
\begin{equation} \label{eq_transport_equation}\tag{$\mathscr T$}
\begin{cases}
\displaystyle{\frac{\partial u}{\partial \tau} - \left\langle \bb,\nabla u \right\rangle + cu = 0} & \text{in $(0,\bar{\tau}) \times \mathbb{G}$}\\
u(0,\cdot)=u_0 & \text{in $\mathbb{G}$,}
\end{cases}
\end{equation}
where a velocity field $\bb\in \mathbb{X}(\mathbb{G})$, a reaction term $c$ and an initial condition $u_0$ are given. Fix $\theta \in [1,\infty]$. From now on, we suppose that
\begin{equation} \label{b,c_assumptions}
\bb\in L^1\left(0,\bar{\tau};\left[L^{\theta'}_{\mathrm{loc}}(\mathbb{G})\right]^n\right), \qquad c,\divv \bb \in L^1\left(0,\bar{\tau};L^{\theta'}_{\mathrm{loc}}(\mathbb{G})\right),\qquad u_0\in L^\theta_{\mathrm{loc}}(\mathbb{G}).
\end{equation}
We emphasize that, in the first assumption above (and similarly in the analogous ones below), the integrability of $\bb$ is understood component-wise with respect to the fixed left-invariant basis. We say that $u \in L^\infty\left(0,\bar{\tau};L^\theta_{\mathrm{loc}}(\mathbb{G})\right)$ ($u\in L^\infty((0,\bar\tau)\times\mathbb{G})$ if $\theta=\infty$) is a \emph{distributional solution to \eqref{eq_transport_equation}} if
\begin{equation*}
 \left\langle \mathscr T_{u_0,\bb,c}(u),\varphi\right\rangle\coloneqq-\int_{\mathbb{G}}u_0(p)\varphi(0,p)\,dp+\int_{0}^{\bar{\tau}}\int_{\mathbb{G}}u\left[-\partial_\tau\varphi+\langle\bb, \nabla \varphi \rangle+(c+\divv\bb)\varphi\right] dp\,d\tau=0 
\end{equation*}
for every $\varphi\in C_c^\infty([0,\bar{\tau})\times\mathbb{G})$. The existence of solutions to \eqref{eq_transport_equation} is ensured by the next result, similar to \cite[Proposition II.1]{MR1022305}.
\begin{proposition} \label{prop_existence}
Let $\theta \in [1,\infty]$ and assume \eqref{b,c_assumptions}. Assume, in addition, that $u_0 \in L^\theta(\mathbb{G})$ and
\begin{equation*}
\begin{cases}
c+\frac{1}{\theta}\divv \bb \in L^1(0,\bar{\tau};L^\infty(\mathbb{G})) & \text{if $\theta \in (1,\infty]$}\\
c, \divv \bb \in L^1(0,\bar{\tau};L^\infty(\mathbb{G})) & \text{if $\theta=1$}.
\end{cases}
\end{equation*}
Then there exists a distributional solution $u \in L^\infty\left(0,\bar{\tau};L^\theta(\mathbb{G})\right)$ to \eqref{eq_transport_equation}.
\end{proposition}
The following property (cf.~\cite{AST_2008__317__175_0}) allows us to handle the initial condition. See \cite[Proposition 4.3]{dpl1} for the proof.
\begin{proposition}\label{propositiondelellis}
  Let $u \in L^\infty\left(0,\bar{\tau};L^\theta_{\mathrm{loc}}(\mathbb{G})\right)$ be a distributional solution to \eqref{eq_transport_equation}. Extend $\bb,c,u$ to $(-\infty,\bar\tau)\times\mathbb{G}$ by setting
   \begin{equation}\label{extensiondelellis}
\bb(\tau,p)=\begin{cases}
0 & \text{if $\tau<0$}\\
\bb(\tau,p) & \text{otherwise,}
\end{cases}
\qquad   
c(\tau,p)=\begin{cases}
0 & \text{if $\tau<0$}\\
c(\tau,p) & \text{otherwise,}
\end{cases}
\qquad 
u(\tau,p)=\begin{cases}
u_0(p) & \text{if $\tau<0$}\\
u(\tau,p) & \text{otherwise.}
\end{cases}
\end{equation}
Then 
\begin{equation*}
    \partial_\tau u - \left\langle \bb,\nabla u\right\rangle + cu = 0  \qquad\text{in $(-\infty,\bar{\tau}) \times \mathbb{G}$}
\end{equation*}
in the sense of distributions, namely $\langle \mathscr T_{\bb,c}(u),\varphi\rangle=0$ for every $\varphi\in C_c^\infty((-\infty,\bar{\tau})\times\mathbb{G})$, where
\begin{equation}\label{equazioneestesasenzadato}
   \left\langle \mathscr T_{\bb,c}(u),\varphi\right\rangle\coloneqq \int_{-\infty}^{\bar{\tau}}\int_{\mathbb{G}}u\left[-\partial_\tau\varphi+\langle \bb, \nabla \varphi \rangle+(c+\divv \bb)\varphi\right] dp\,d\tau.
\end{equation}
\end{proposition}
We recall the notion of renormalized solution, introduced by DiPerna and Lions in \cite{MR1022305}. A function $u\in L^\infty\left(0,\bar\tau;L^\theta(\mathbb{G})\right)$ is a \emph{renormalized solution} to \eqref{eq_transport_equation} if, for every $\beta \in C^1(\rr)$ with
$\beta'$ bounded, the function $\beta(u)$ is a distributional solution to
\begin{equation*} 
\begin{cases}
\dfrac{\partial \beta(u)}{\partial \tau}-\left\langle \bb,\nabla \beta (u)\right\rangle + cu\beta' (u) = 0 & \text{in $(0,\bar{\tau}) \times\mathbb{G}$}\\
\beta(u)(0,\cdot)=\beta(u_0) & \text{in $\mathbb{G}$,}
\end{cases}
\end{equation*}
namely if, for every $\varphi\in C^\infty_c\left([0,\bar\tau)\times\mathbb{G}\right)$,
\begin{equation*}
    -\int_{\mathbb{G}}\beta(u_0)(p)\varphi(0,p)\,dp
    +\int_{0}^{\bar{\tau}}\int_{\mathbb{G}}\beta(u)\left[-\partial_\tau \varphi+\left\langle \bb, \nabla \varphi\right \rangle+\divv \bb\,\varphi\right]+c u\beta'(u)\varphi\,dp\,d\tau =0.
\end{equation*}
If $\theta=\infty$, the condition that $\beta'$ is bounded is not required. Choosing $\beta$ to be the identity map, it is clear that renormalized solutions are distributional solutions. The converse implication does not hold in general (see \cite{DePauw_03}). However, if it holds for all $u_0$ and $c$ as in \eqref{b,c_assumptions}, we say that $\bb$ enjoys the \emph{renormalization property}. We show that contact vector fields with horizontal Sobolev regularity have the renormalization property.
\begin{theorem} \label{teo_distributional_implies_renormalized}
Let $\theta \in [1,\infty]$ and $u_0 \in L^\theta(\mathbb{G})$. Assume that $\bb \in L^1\left(0,\bar{\tau};\left[W^{1,\theta'}_{\hhh,\mathrm{loc}}(\mathbb{G})\right]^n\right)$ is a time-dependent contact vector field and $c \in L^1\left(0,\bar{\tau};L^{\theta'}_{\mathrm{loc}}(\mathbb{G})\right)$. Then any distributional solution $u \in L^\infty\left(0,\bar{\tau};L^\theta(\mathbb{G})\right)$ to \eqref{eq_transport_equation} 
is a renormalized solution. In particular, $\bb$ has the renormalization property.
\end{theorem}

\subsection{Regularization}\label{sec_refgogogog}
The proof of \Cref{teo_distributional_implies_renormalized} follows the mollification  scheme of \cite{MR1022305}: mollifying a distributional solution to the transport equation yields a smooth (in space) solution up to an error term, known as commutator. Here, mollification is meant with respect to the group convolution as in \Cref{subsec_groupconv}. 
The next result provides an integral representation of the commutator. Its proof is analogous to \cite[Proposition 4.9]{dpl1}.
\begin{proposition}\label{prop_formadelresto}
 For $\theta \in [1,\infty]$, assume \eqref{b,c_assumptions}. Fix a distributional solution $u \in L^\infty\left(0,\bar{\tau}; L^\theta(\mathbb{G})\right)$ to \eqref{eq_transport_equation} and extend $\bb,c,u$ to $(-\infty,\bar\tau) \times \mathbb{G}$ as in \eqref{extensiondelellis}. For every $\varepsilon>0$, let $u_\varepsilon = u*\rho_\varepsilon$ be the spatial group mollification of $u$ by $\rho_\varepsilon$, with $\rho_\varepsilon$ as in \eqref{eq_def_mollifier}. Then, the distribution $\mathscr T_{\bb,c}(u_\varepsilon)$ in \eqref{equazioneestesasenzadato} is representable by integration of the \emph{commutator} $\mathscr C_\varepsilon$, i.e.,
  \begin{equation*}
        \left\langle \mathscr T_{\bb,c}(u_\varepsilon),\varphi\right\rangle=\int_{-\infty}^{\bar\tau}\int_{\mathbb{G}}\varphi(\tau,p)\mathscr C_\varepsilon(\tau,p)\,dp\,d\tau\qquad\text{for every $\varphi\in C^\infty_c((-\infty,\bar\tau)\times\mathbb{G})$,}
  \end{equation*}
where
$\mathscr C_\varepsilon=\mathscr C_\varepsilon^1+\mathscr C_\varepsilon^2\in L^1\left(-\infty,\bar\tau,L_{\mathrm{loc}}^1(\mathbb{G})\right)$ is defined by 
\begin{align*}
    \mathscr C_\eps^1(\tau,p)&\coloneqq -\left((u\divv \bb)*\rho_\varepsilon\right)(\tau, p)-\int_{\mathbb{G}}u\left(\tau, p\cdot q^{-1} \right)\sum_{i=1}^s \sum_{\alpha=1}^{n_i} \left[b^i_\alpha(\tau, p) X^i_\alpha \rho_\varepsilon(q)-b^i_\alpha \left(\tau, p\cdot q^{-1} \right) \left(X^i_\alpha\right)^r \rho_\varepsilon(q)\right] dq,\\
     \mathscr C_\eps^2(\tau,p)&\coloneqq \int_{\mathbb{G}}u\left(\tau, p\cdot q^{-1}\right )\rho_\varepsilon(q)\left[c(\tau, p )-c\left(\tau, p\cdot q^{-1}\right)\right] dq
\end{align*}
for a.e.~$(\tau,p)\in(-\infty,\bar\tau)\times\mathbb{G}$.
\end{proposition}

The following lemma constitutes the crucial step in the proof of \Cref{teo_distributional_implies_renormalized}:  the commutator vanishes in the limit, provided that the velocity field is a contact vector field. Recall that, if $\bb \in \left[W^{1,\theta}_{\hhh,\mathrm{loc}}(\mathbb{G})\right]^n$ is a contact vector field, then \eqref{eq_iterated_contact_components} holds. Recall also the notation \eqref{eq_notation_multi-index} and \eqref{eq_notation_multi-index_2}.
\begin{lemma} \label{lem_commutator}
Let $\theta \in [1,\infty]$. Fix $u_0 \in L^\theta_{\mathrm{loc}}(\mathbb{G})$. Let $\bb \in L^1\left(0,\bar{\tau};\left[W^{1,\theta'}_{\hhh,\mathrm{loc}}(\mathbb{G})\right]^n\right)$ be a time-dependent contact vector field and $c \in L^1\left(0,\bar{\tau};L^{\theta'}_{\mathrm{loc}}(\mathbb{G})\right)$. For $u \in L^\infty\left(0,\bar{\tau}; L^\theta_{\mathrm{loc}}(\mathbb{G})\right)$, extend $\bb,c,u$ to $(-\infty,\bar\tau) \times \mathbb{G}$ as in \eqref{extensiondelellis}. Then  $\mathscr C_\varepsilon \to 0$ in $L^1\left(-\infty,\bar\tau; L^1_{\mathrm{loc}}(\mathbb{G})\right)$ as $\varepsilon \searrow 0$. 
\end{lemma}
\begin{proof}
Every argument will be carried out for $\tau$ fixed, so we assume without loss of generality that $u,\bb$ and $c$ do not depend on $\tau$. We only prove the local $L^1$-convergence of $\mathscr C^1_\varepsilon$, since the proof for $\mathscr C^2_\varepsilon$ is identical to the corresponding case in \cite[Lemma 4.10]{dpl1}. For $p \in \mathbb{G}$, one has
\begin{equation*}
\begin{split}
    -&\mathscr C_\eps^1(p) = \left((u\divv \bb)*\rho_\varepsilon\right)(p)+\int_{\mathbb{G}}u\left(p\cdot q^{-1} \right)\sum_{i=1}^s \sum_{\alpha=1}^{n_i} \left[b^i_\alpha(p) X^i_\alpha \rho_\varepsilon(q)-b^i_\alpha \left(p\cdot q^{-1} \right) \left(X^i_\alpha\right)^r \rho_\varepsilon(q)\right] dq\\
    \overset{\eqref{eq_right_vector_fields}}&{=} \left((u\divv \bb)*\rho_\varepsilon\right)(p)+ \sum_{i=1}^s \sum_{\alpha=1}^{n_i} \biggl\{ \int_{\mathbb{G}} \left[b^i_\alpha(p) - b^i_\alpha \left(p\cdot q^{-1} \right) \right] u\left(p\cdot q^{-1} \right) \left(X^i_\alpha\right)^r \rho_\varepsilon(q)\,dq\\
    & + \sum_{k=1}^{s-i} \frac{1}{k!} \sum_{\abs*{I_k} \leq s-i} \sum_{\mathcal A_k(I_k)} \sum_{\mathcal B_k(i,I_k)} c(I_k,i+\abs*{I_k})_{A_k B_{k-1}}^{B_k} \int_{\mathbb{G}} b^i_\alpha(p) u\left(p \cdot q^{-1}\right) q^{I_k}_{A_k} \left(X^{i+\abs*{I_k}}_{\beta_k}\right)^r \rho_\eps(q)\,dq \biggr\}\\
    \overset{\eqref{eq_derivative_mollification}}&{=} \left((u\divv \bb)*\rho_\varepsilon\right)(p)+ \sum_{i=1}^s \sum_{\alpha=1}^{n_i} \biggl\{ \int_{\mathbb{G}} \frac{b^i_\alpha(p) - b^i_\alpha \left(p\cdot q^{-1} \right)}{\eps^{Q+i}} u\left(p\cdot q^{-1} \right) \left(\left(X^i_\alpha\right)^r \rho\right)\left(\delta_{\frac{1}{\eps}}(q)\right) dq\\
    & + \sum_{k=1}^{s-i} \frac{1}{k!} \sum_{\abs*{I_k} \leq s-i} \sum_{\mathcal A_k(I_k)} \sum_{\mathcal B_k(i,I_k)} c(I_k,i+\abs*{I_k})_{A_k B_{k-1}}^{B_k} \int_{\mathbb{G}} \frac{b^i_\alpha(p)}{\eps^{Q+i+\abs*{I_k}}} u\left(p \cdot q^{-1}\right) q^{I_k}_{A_k} \left(\left(X^{i+\abs*{I_k}}_{\beta_k}\right)^r \rho\right)\left(\delta_{\frac{1}{\eps}}(q)\right) dq \biggr\}.
\end{split}
\end{equation*}
We perform the change of variables $w = \delta_{\frac{1}{\varepsilon}}(q)$. Observe that
\begin{equation*}
q^{I_k}_{A_k}=q^{i_1}_{\alpha_1} \dots q^{i_k}_{\alpha_k} = \eps^{i_1+\dots+i_k} w^{i_1}_{\alpha_1} \dots w^{i_k}_{\alpha_k}=\eps^{\abs*{I_k}}w^{I_k}_{A_k}.
\end{equation*}
Since $\left(\left(X^i_\alpha \right)^r\rho\right)\left(w^{-1}\right)=-X^i_\alpha \rho(w)$ for every $i=1,\ldots,s$ and $\alpha=1,\ldots,n_i$ by \Cref{lemscambiarederivateeinversione} and \eqref{eq_mollificatori}, then
\begin{equation*}
\begin{split}
    -\mathscr C_\eps^1(p) \overset{\eqref{eq_dilation_Lebesgue}}&{=} \left((u\divv \bb)*\rho_\varepsilon\right)(p)+ \sum_{i=1}^s \sum_{\alpha=1}^{n_i} \biggl\{ \int_{\mathbb{G}} \frac{b^i_\alpha(p) - b^i_\alpha \left(p\cdot \delta_\eps\left(w^{-1}\right) \right)}{\eps^i} u\left(p\cdot \delta_\eps\left(w^{-1}\right) \right) \left(X^i_\alpha\right)^r \rho(w)\,dw\\
    & + \sum_{k=1}^{s-i} \frac{1}{k!} \sum_{\abs*{I_k} \leq s-i} \sum_{\mathcal A_k(I_k)} \sum_{\mathcal B_k(i,I_k)} c(I_k,i+\abs*{I_k})_{A_k B_{k-1}}^{B_k} \int_{\mathbb{G}} \frac{b^i_\alpha(p)}{\eps^i} u\left(p \cdot \delta_\eps\left(w^{-1}\right) \right) w^{I_k}_{A_k} \left(X^{i+\abs*{I_k}}_{\beta_k}\right)^r \rho(w)\,dw \biggr\}\\
    & = \left((u\divv \bb)*\rho_\varepsilon\right)(p)+ \sum_{i=1}^s \sum_{\alpha=1}^{n_i} \biggl\{ \underbrace{\int_{\mathbb{G}} \frac{b^i_\alpha (p\cdot \delta_\eps(w)) - b^i_\alpha(p)}{\eps^i} u(p\cdot \delta_\eps(w)) X^i_\alpha \rho(w)\,dw}_{\eqqcolon C^i_\alpha(p)}\\
    & \underbrace{-\sum_{k=1}^{s-i} \frac{1}{k!} \sum_{\abs*{I_k} \leq s-i} \sum_{\mathcal A_k(I_k)} \sum_{\mathcal B_k(i,I_k)} (-1)^k c(I_k,i+\abs*{I_k})_{A_k B_{k-1}}^{B_k} \int_{\mathbb{G}} \frac{b^i_\alpha(p)}{\eps^i} u(p \cdot \delta_\eps(w)) w^{I_k}_{A_k} X^{i+\abs*{I_k}}_{\beta_k} \rho(w)\,dw}_{\eqqcolon D^i_\alpha(p)} \biggr\}.
\end{split}
\end{equation*}
First, we focus on $C^i_\alpha$. Recalling \eqref{eq_Taylor_polynomial}, we have
\begin{equation*}
\begin{split}
C^i_\alpha(p) & = \underbrace{\int_{\mathbb{G}} \frac{b^i_\alpha(p \cdot \delta_\eps(w))-P^{i-1}_p b^i_\alpha (\delta_\eps(w))}{\eps^i} u(p \cdot \delta_\eps(w)) X^i_\alpha \rho(w)\,dw}_{\eqqcolon F^i_\alpha(p)} \\
& \quad + \underbrace{\int_{\mathbb{G}} \sum_{\ell=1}^{i-1} T^\ell_p b^i_\alpha(\delta_\eps(w)) \eps^{-i} u(p \cdot \delta_\eps(w)) X^i_\alpha \rho(w)\,dw}_{\eqqcolon G^i_\alpha(p)},
\end{split}
\end{equation*}
where we agree that, when $i=1$, $P^0_p b^1_\alpha(\delta_\eps(w)) = b^1_\alpha(p)$ and $G^1_\alpha(p)=0$ for every $\alpha=1,\dots,n_1$. We compute the $L^1_{\mathrm {loc}}$-limit of $F^i_\alpha$. Assume first $\theta'<\infty$. Since $\bb$ is a contact vector field, $b^i_\alpha \in W^{i,\theta'}_{\hhh,\mathrm{loc}}(\mathbb{G})$ by \Cref{rem_iterated_contact_components}. Recalling \eqref{eq_higher_order_derivatives} and \eqref{eq_difference_quotients_notation}, the dominated convergence theorem and \Cref{thm_difference_quotients} imply that, for every $K \Subset \mathbb{G}$,
\begin{equation*}
\begin{split}
    \int_K & \abs*{\int_{\mathbb{G}} \left[\frac{b^i_\alpha(p \cdot \delta_\eps(w))-P^{i-1}_p b^i_\alpha (\delta_\eps(w))}{\eps^i} u(p \cdot \delta_\eps(w)) - T^i_p b^i_\alpha(w) u(p) \right] X^i_\alpha \rho(w)\,dw}\,dp\\
    & \leq \int_{\mathbb{G}} \abs*{X^i_\alpha \rho(w)} \left(\int_K \abs*{\frac{b^i_\alpha(p \cdot \delta_\eps(w))-P^{i-1}_p b^i_\alpha (\delta_\eps(w))}{\eps^i} u(p \cdot \delta_\eps(w))- T^i_p b^i_\alpha(w) u(p)} \,dp \right)\,dw \xrightarrow[\varepsilon \searrow 0]{} 0,
\end{split}
\end{equation*}
where we also used the $L^\theta$-continuity of right translations (\Cref{lem_continuity_translations}) to replace $u(p \cdot \delta_\varepsilon(w))$ with $u(p)$ before applying \Cref{thm_difference_quotients}. If instead $\theta'=\infty$, the argument is similar to the corresponding case in the proof of \cite[Lemma 4.10]{dpl1}.
Hence,
\begin{equation*}
\begin{split}
    \lim_{\eps \searrow 0}^{L^1_{\mathrm{loc}}} \sum_{i=1}^s \sum_{\alpha=1}^{n_i} F^i_\alpha(p) & = \sum_{i=1}^s \sum_{\alpha=1}^{n_i} \int_{\mathbb{G}} T^i_p b^i_\alpha(w) u(p) X^i_\alpha \rho(w) \,dw\\
    \overset{\eqref{eq_higher_order_derivatives}}&{=} \sum_{i=1}^s \sum_{\alpha=1}^{n_i} \sum_{j=1}^i \frac{1}{j!} \sum_{\substack{i_1,\dots,i_j=1 \\ i_1+\dots+i_j=i}}^s \sum_{\alpha_1=1}^{n_{i_1}} \dots \sum_{\alpha_j=1}^{n_{i_j}} X^{i_1}_{\alpha_1} \dots X^{i_j}_{\alpha_j} b^i_\alpha(p) u(p) \int_{\mathbb{G}} w^{i_1}_{\alpha_1} \dots w^{i_j}_{\alpha_j} X^i_\alpha \rho(w) \,dw\\
    \overset{\eqref{eq_derivatives_coordinates}}&{=} \sum_{i=1}^s \sum_{\alpha=1}^{n_i} \Biggl\{\sum_{\alpha_1=1}^{n_i} X^i_{\alpha_1} b^i_\alpha(p) u(p) \int_{\mathbb{G}} w^i_{\alpha_1} X^i_{\alpha} \rho(w) \,dw\\
    & \quad + \sum_{j=2}^i \frac{1}{j!} \sum_{\substack{i_1,\dots,i_j=1 \\ i_1+\dots+i_j=i}}^s \sum_{\alpha_1=1}^{n_{i_1}} \dots \sum_{\alpha_j=1}^{n_{i_j}} X^{i_1}_{\alpha_1} \dots X^{i_j}_{\alpha_j} b^i_\alpha(p) u(p) \underbrace{\int_{\mathbb{G}} X^i_\alpha \left(w^{i_1}_{\alpha_1} \dots w^{i_j}_{\alpha_j}\rho(w)\right) dw}_{=0}\Biggr\}\\
    \overset{\eqref{eq_derivatives_coordinates}}&{=} \sum_{i=1}^s \sum_{\alpha=1}^{n_i} \sum_{\substack{\alpha_1=1 \\ \alpha_1 \neq \alpha}}^{n_i} X^i_{\alpha_1} b^i_\alpha(p) u(p) \underbrace{\int_{\mathbb{G}} X^i_{\alpha} \left(w^i_{\alpha_1}\rho(w)\right) dw}_{=0} + \sum_{i=1}^s \sum_{\alpha=1}^{n_i} X^i_{\alpha} b^i_\alpha(p) u(p) \int_{\mathbb{G}} w^i_{\alpha} X^i_{\alpha} \rho(w) \,dw\\
    & = \sum_{i=1}^s \sum_{\alpha=1}^{n_i} X^i_{\alpha} b^i_\alpha(p) u(p) \biggl(\underbrace{\int_{\mathbb{G}} X^i_{\alpha} \left(w^i_{\alpha}\rho(w)\right) dw}_{=0} - \int_{\mathbb{G}} \rho(w) \underbrace{X^i_\alpha w^i_{\alpha}}_{\overset{\eqref{eq_derivatives_coordinates}}{=}1} \,dw\biggr)\\
    & = -u(p) \sum_{i=1}^s \sum_{\alpha=1}^{n_i} X^i_{\alpha} b^i_\alpha(p)\\
    \overset{\eqref{eq_divergence}}&{=} -u(p) \divv \bb(p).
\end{split}
\end{equation*}
It remains to analyze $D^i_\alpha(p)+G^i_\alpha(p)$. Observe that
\begin{equation*}
\begin{split}
    \sum_{i=1}^s & \sum_{\alpha=1}^{n_i} \left[D^i_\alpha(p)+G^i_\alpha(p)\right] \\
    & = -\sum_{i=1}^{s-1} \sum_{\alpha=1}^{n_i} \sum_{k=1}^{s-i} \frac{1}{k!} \sum_{\abs*{I_k} \leq s-i} \sum_{\mathcal A_k(I_k)} \sum_{\mathcal B_k(i,I_k)} (-1)^k c(I_k,i+\abs*{I_k})_{A_k B_{k-1}}^{B_k} \int_{\mathbb{G}} \frac{b^i_\alpha(p)}{\eps^i} u(p \cdot \delta_\eps(w)) w^{I_k}_{A_k} X^{i+\abs*{I_k}}_{\beta_k} \rho(w)\,dw\\
    & \quad + \sum_{i=2}^s \sum_{\alpha=1}^{n_i} \sum_{\ell=1}^{i-1} \int_{\mathbb{G}} T^\ell_p b^i_\alpha(\delta_\eps(w)) \eps^{-i} u(p \cdot \delta_\eps(w)) X^i_\alpha \rho(w)\,dw.
\end{split}
\end{equation*}
For the first term, we get
\begin{equation*}
\begin{split}
    & \sum_{i=1}^{s-1} \sum_{\alpha=1}^{n_i} \sum_{k=1}^{s-i} \frac{1}{k!} \sum_{\abs*{I_k} \leq s-i} \sum_{\mathcal A_k(I_k)} \sum_{\mathcal B_k(i,I_k)} (-1)^k c(I_k,i+\abs*{I_k})_{A_k B_{k-1}}^{B_k} \int_{\mathbb{G}} \frac{b^i_\alpha(p)}{\eps^i} u(p \cdot \delta_\eps(w)) w^{I_k}_{A_k} X^{i+\abs*{I_k}}_{\beta_k} \rho(w)\,dw\\
    \overset{\eqref{eq_skew_structural_constants}}&{=} \sum_{i=1}^{s-1} \sum_{\alpha=1}^{n_i} \sum_{k=1}^{s-i} \frac{1}{k!} \sum_{\abs*{I_k} \leq s-i} \sum_{\mathcal A_k(I_k)} \sum_{\mathcal B_k(i,I_k)} c(i+\abs*{I_k},I_k)_{B_{k-1} A_k}^{B_k} \frac{b^i_\alpha(p)}{\eps^{i+\abs*{I_k}}} \int_{\mathbb{G}} u(p \cdot \delta_\eps(w)) \delta_\eps\left(w^{I_k}_{A_k}\right) X^{i+\abs*{I_k}}_{\beta_k} \rho(w)\,dw\\
    & = \sum_{i=1}^{s-1} \sum_{\alpha=1}^{n_i} \sum_{\ell=1}^{s-i} \sum_{k=1}^\ell \frac{1}{k!} \sum_{\abs*{I_k}=\ell} \sum_{\mathcal A_k(I_k)} \sum_{\mathcal B_k(i,I_k)} c(i+\ell,I_k)_{B_{k-1} A_k}^{B_k} \frac{b^i_\alpha(p)}{\eps^{i+\ell}} \int_{\mathbb{G}} u(p \cdot \delta_\eps(w)) \delta_\eps\left(w^{I_k}_{A_k}\right) X^{i+\ell}_{\beta_k} \rho(w)\,dw\\
    & = \sum_{i=2}^s \sum_{\ell=1}^{i-1} \sum_{k=1}^\ell \frac{1}{k!} \sum_{\abs*{I_k}=\ell} \sum_{\mathcal A_k(I_k)} \sum_{\mathcal B_k(i-\ell,I_k)} \sum_{\alpha=1}^{n_{i-\ell}} c(i,I_k)_{B_{k-1} A_k}^{B_k} \frac{b^{i-\ell}_\alpha(p)}{\eps^i} \int_{\mathbb{G}} \delta_\eps\left(w^{I_k}_{A_k}\right) u(p \cdot \delta_\eps(w)) X^i_{\beta_k} \rho(w)\,dw\\
    & = \sum_{i=2}^s \sum_{\ell=1}^{i-1} \sum_{k=1}^\ell \frac{1}{k!} \sum_{\abs*{I_k}=\ell} \sum_{\mathcal A_k(I_k)} \sum_{\beta_k=1}^{n_i} \sum_{\mathcal B_{k-1}(i-\ell,I_{k-1})} \sum_{\alpha=1}^{n_{i-\ell}} c(i,I_k)_{B_{k-1} A_k}^{B_k} \frac{b^{i-\ell}_\alpha(p)}{\eps^i} \int_{\mathbb{G}} \delta_\eps\left(w^{I_k}_{A_k}\right) u(p \cdot \delta_\eps(w)) X^i_{\beta_k} \rho(w) \,dw.
\end{split}
\end{equation*}
Instead, the second one is
\begin{equation*}
\begin{split}
    \sum_{i=2}^s & \sum_{\alpha=1}^{n_i} \sum_{\ell=1}^{i-1} \int_{\mathbb{G}} T^\ell_p b^i_\alpha(\delta_\eps(w)) \eps^{-i} u(p \cdot \delta_\eps(w)) X^i_\alpha \rho(w)\,dw\\
    \overset{\eqref{eq_higher_order_derivatives}}&{=} \sum_{i=2}^s \sum_{\alpha=1}^{n_i} \sum_{\ell=1}^{i-1} \sum_{k=1}^\ell \frac{1}{k!} \sum_{\abs*{I_k}=\ell} \sum_{\mathcal A_k(I_k)} \int_{\mathbb{G}} \delta_\eps\left(w^{I_k}_{A_k}\right) X^{I_k}_{A_k} b^i_\alpha(p)\eps^{-i} u(p \cdot \delta_\eps(w)) X^i_\alpha \rho(w) \,dw\\
    & = \sum_{i=2}^s \sum_{\ell=1}^{i-1} \sum_{k=1}^\ell \frac{1}{k!} \sum_{\abs*{I_k}=\ell} \sum_{\mathcal A_k(I_k)} \sum_{\beta_k=1}^{n_i} X^{I_k}_{A_k} b^i_{\beta_k}(p) \int_{\mathbb{G}} \delta_\eps\left(w^{I_k}_{A_k}\right) u(p \cdot \delta_\eps(w)) \eps^{-i} X^i_{\beta_k} \rho(w) \,dw.
\end{split}
\end{equation*}
Therefore,
\begin{equation*}
\begin{split}
    & \sum_{i=1}^s \sum_{\alpha=1}^{n_i} \left[D^i_\alpha(p)+G^i_\alpha(p)\right]\\
    & = \sum_{i=2}^s \sum_{\ell=1}^{i-1} \sum_{k=1}^\ell \frac{1}{k!} \sum_{\abs*{I_k}=\ell} \sum_{\mathcal A_k(I_k)} \sum_{\beta_k=1}^{n_i} \\
    & \quad \left[X^{I_k}_{A_k} b^i_{\beta_k}(p) - \sum_{\mathcal B_{k-1}(i-\ell,I_{k-1})} \sum_{\alpha=1}^{n_{i-\ell}} c(i,I_k)_{B_{k-1} A_k}^{B_k} b^{i-\ell}_\alpha(p)\right] \int_{\mathbb{G}} \delta_\eps\left(w^{I_k}_{A_k}\right) u(p \cdot \delta_\eps(w)) \eps^{-i} X^i_{\beta_k} \rho(w) \,dw\\
    \overset{\eqref{eq_iterated_contact_components}}&{=} 0.
\end{split}
\end{equation*}
In conclusion,
\begin{equation*}
\lim_{\eps \searrow 0}^{L^1_{\mathrm{loc}}}  -\mathscr C_\eps^1 = \lim_{\varepsilon \searrow 0}^{L^1_{\mathrm{loc}}} (u\divv \bb)*\rho_\varepsilon + \lim_{\varepsilon \searrow 0}^{L^1_{\mathrm{loc}}} \left(\sum_{i=1}^s \sum_{\alpha=1}^{n_i} F^i_\alpha \right) + \lim_{\varepsilon \searrow 0}^{L^1_{\mathrm{loc}}} \Biggl(\underbrace{\sum_{i=1}^s \sum_{\alpha=1}^{n_i} \left[D^i_\alpha + G^i_\alpha \right]}_{=0}\Biggr)= u\divv \bb-u\divv \bb=0,
\end{equation*}
where the second equality holds due to \Cref{prop_group_mollification}, since $u\divv \bb \in L^1_{\mathrm{loc}}(\mathbb{G})$.
\end{proof}
\begin{proof}[Proof of \Cref{teo_distributional_implies_renormalized}]
Owing to \Cref{lem_commutator}, the proof is just an adaptation of that of \cite[Theorem 4.6]{dpl1}.
\end{proof}

\subsection{Well-posedness results}\label{section_wpr}
As a consequence of \Cref{teo_distributional_implies_renormalized} and \Cref{prop_existence}, we get existence and uniqueness of distributional solutions to \eqref{eq_transport_equation} under natural growth conditions. We refer to \cite[Theorem 4.7]{dpl1} for its proof in the Heisenberg groups case and \cite[Theorem II.2]{MR1022305} in the Euclidean one.
\begin{theorem}\label{exun_pde_theorem}
Let $\theta \in [1,\infty]$ and $u_0 \in L^\theta(\mathbb{G})$. Let $\bb \in L^1\left(0,\bar{\tau};\left[W^{1,\theta'}_{\hhh,\mathrm{loc}}(\mathbb{G})\right]^n\right)$ be a time-dependent contact vector field. Assume, in addition, that $c, \divv \bb \in L^1\left(0,\bar{\tau};L^\infty(\mathbb{G})\right)$ and that
\begin{equation} \label{eq_growth_condition}
\frac{|\bb|}{1+d(p,0)} \in L^1\left(0,\bar{\tau};L^1(\mathbb{G})\right)+L^1(0,\bar{\tau};L^\infty(\mathbb{G})),
\end{equation} 
where $|\bb|$ is the norm of $\bb$ with respect to the Riemannian metric $\langle\cdot,\cdot\rangle$. Then there exists a unique distributional solution $u$ to \eqref{eq_transport_equation} in $L^\infty\left(0,\bar{\tau};L^\theta(\mathbb{G})\right)$ corresponding to the initial condition $u_0$.
\end{theorem}
A slight adaptation of the abstract framework introduced in \cite{MR2096794} -- and subsequently formalized in \cite{MR2409676,MR3283066} -- allows us to leverage Eulerian well-posedness to well-posedness at Lagrangian level.
Accordingly, we recall that a map $\Phi(\tau,p)$ is a \emph{regular Lagrangian flow} associated with $\bb \in L^1(0,\bar{\tau};\mathbb{X}(\mathbb{G}))$ if
\begin{itemize}
\item[(a)] for a.e.~$p \in \mathbb{G}$, $\tau\mapsto\Phi(\tau,p)$ is an absolutely continuous integral curve of $\bb$, i.e.,
\begin{equation*}
\begin{cases}
\dfrac{d}{d\tau}\Phi(\tau,p)=\bb(\tau,\Phi(\tau,p)) & \text{for a.e.~$\tau \in (0,\bar{\tau})$}\\
\Phi(0,p)=p,
\end{cases}
\end{equation*}
\item[(b)] there exists a constant $C\in (0,\infty)$ such that $\Phi(\tau,\cdot)_\#\mathcal{L}^n\leq C\mathcal{L}^n$ for all $\tau \in [0,\bar\tau)$.
\end{itemize}
In the above definition,  $\mathcal{L}^n$ is the Lebesgue measure and $\Phi(\tau,\cdot)_\#\mathcal{L}^n$ the push-forward measure.
\begin{theorem}
Assume that $\bb \in L^1\left(0,\bar{\tau};\left[W^{1,1}_{\hhh,\mathrm{loc}}(\mathbb{G})\right]^n\right)$ is a time-dependent contact vector field. Assume that $\divv \bb \in L^1\left(0,\bar{\tau};L^\infty(\mathbb{G})\right)$ and that \eqref{eq_growth_condition} holds. Then there exists a regular Lagrangian flow associated with $\bb$. Moreover, if $\Phi_1$ and $\Phi_2$ are regular Lagrangian flows associated with $\bb$, then
\begin{equation*}
    \Phi_1(\cdot,p)=\Phi_2(\cdot,p)\qquad\text{for a.e.~$p\in \mathbb G$.}
\end{equation*}
\end{theorem}
\begin{remark}
    Stability properties, both at Eulerian and at Lagrangian level, are a direct consequence of uniqueness as in \Cref{exun_pde_theorem}, and follows with minor modifications as in \cite[Section $\mathrm{II}$]{MR1022305} and \cite[Section 5]{MR2409676}.
\end{remark}
We refer to \cite[Section 5]{dpl1} for a summary of further consequences in the case of Heisenberg groups, and for several additional references. With small adjustments, these results carry over to the Carnot group setting.

\appendix
\section{Further properties of difference quotients}\label{sec_appendix}
In this final appendix, we collect some additional properties of difference quotients. Although they are not needed for the purposes of this paper, we believe they may be of independent interest. The first result provides an inductive characterization of difference quotients.
\begin{proposition}
Let $\Om \subseteq \mathbb{G}$ be open. Fix $p \in \Om$ and $w \in \mathbb{G}$. Let $\eps>0$ be such that $p \cdot \delta_{\tau\eps}(w) \in \Om$ for every $\tau \in [0,1]$. If $w=\exp(W)=\exp(W_1+\dots+W_s)$ with $W_i \in V_i$ for every $i=1,\dots,s$, set 
\begin{equation*}
D^w_m=\sum_{k=0}^{s-1} \frac{(-1)^k}{(k+1)!} \sum_{\abs*{I_{k+1}}=m} i_1 [W_{i_{k+1}},\dots,[W_{i_2},W_{i_1}]\dots].
\end{equation*}
 Let $\ell \geq 1$ and $f \in C^\ell(\Om)$. Then
\begin{equation*}
R^w_{\ell,\eps} f(p)=\sum_{m=1}^\ell \int_0^1 \tau^{\ell-1} R^w_{\ell-m,\tau \eps}(D^w_m f)(p) \,d\tau + \sum_{m=\ell+1}^s \eps^{m-\ell} \int_0^1 \tau^{m-1} D^w_m f(p \cdot \delta_{\tau\eps}(w)) \,d\tau,
\end{equation*}
where we agree that the second term does not appear when $\ell\geq s$.

\end{proposition}
\begin{proof}
We know from \cite[Theorem 2.14.3]{MR746308} that, for every $Z \in \mathfrak{g}$,
\begin{equation} \label{eq_differential_exponential}
(d\exp)_Z=\sum_{k=0}^{s-1} \frac{(-1)^k}{(k+1)!}(\mathrm{ad}_Z)^k.
\end{equation}
Setting $\gamma(\tau)=p \cdot \delta_\tau(w)=p \cdot \exp(\delta_\tau W)$ for every $\tau \in [0,\eps]$, we have
\begin{equation*}
\begin{split}
    \dot{\gamma}(\tau) & = (dL_p)_{\exp(\delta_\tau W)}(d\exp)_{\delta_\tau W}\left(\frac{d}{d\tau} \delta_\tau W\right)\\
    & = (dL_p)_{\exp(\delta_\tau W)}(d\exp)_{\delta_\tau W}\left(\sum_{i_1=1}^s i_1 \tau^{i_1-1} W_{i_1}(\delta_\tau W)\right)\\
    \overset{\eqref{eq_differential_exponential}}&{=} \sum_{k=0}^{s-1} \frac{(-1)^k}{(k+1)!} \sum_{i_1=1}^s i_1 \tau^{i_1-1} (dL_p)_{\exp(\delta_\tau W)} (\mathrm{ad}_{\delta_\tau W})^k(W_{i_1}(\delta_\tau W))\\
    & = \sum_{k=0}^{s-1} \frac{(-1)^k}{(k+1)!} \sum_{i_1=1}^s i_1 \tau^{i_1-1} \sum_{i_2,\dots,i_{k+1}=1}^s \tau^{i_2+\dots+i_{k+1}} [W_{i_{k+1}},\dots,[W_{i_2},W_{i_1}]\dots](\gamma(\tau))\\
    & = \sum_{k=0}^{s-1} \frac{(-1)^k}{(k+1)!} \sum_{\abs*{I_{k+1}} \leq s} i_1 \tau^{\abs*{I_{k+1}}-1} [W_{i_{k+1}},\dots,[W_{i_2},W_{i_1}]\dots](\gamma(\tau))\\
    & = \sum_{k=0}^{s-1} \frac{(-1)^k}{(k+1)!} \sum_{m=1}^s \sum_{\abs*{I_{k+1}}=m} i_1 \tau^{m-1} [W_{i_{k+1}},\dots,[W_{i_2},W_{i_1}]\dots](\gamma(\tau))\\
    & = \sum_{m=1}^s \tau^{m-1} D^w_m(\gamma(\tau)).
\end{split}
\end{equation*}
Therefore,
\begin{equation} \label{eq_implicit_derivative}
\frac{d}{d\tau}f(p \cdot \delta_\tau(w))=\sum_{m=1}^s \tau^{m-1} D^w_m f(p \cdot \delta_\tau(w))
\end{equation}
and so
\begin{equation*}
\begin{split}
    R^w_{\ell,\eps}f(p) \overset{\eqref{eq_difference_quotients_notation},\eqref{eq_Taylor_polynomial}}&{=} \frac{1}{\eps^\ell} \left[f(p \cdot \delta_\eps(w))-f(p)-\sum_{k=1}^\ell T^k_p f(\delta_\eps(w))\right]\\
    \overset{\eqref{eq_higher_order_derivatives}}&{=} \frac{1}{\eps^\ell} \left[\int_0^1 \frac{d}{d\tau} f(p \cdot \delta_{\tau\eps}(w)) \,d\tau -\sum_{k=1}^\ell \eps^k T^k_p f(w)\right]\\
    \overset{\eqref{eq_implicit_derivative}}&{=} \underbrace{\frac{1}{\eps^\ell} \left[\sum_{m=1}^\ell \int_0^1 \tau^{m-1} \eps^m D^w_m f(p \cdot \delta_{\tau\eps}(w)) \,d\tau -\sum_{k=1}^\ell \eps^k T^k_p f(w)\right]}_{\eqqcolon F}\\
    & \quad +\sum_{m=\ell+1}^s \eps^{m-\ell} \int_0^1 \tau^{m-1} D^w_m f(p \cdot \delta_{\tau\eps}(w)) \,d\tau.
\end{split}
\end{equation*}
The second term appears in the desired formula, so we focus on the first one. We have
\begin{equation*}
\begin{split}
     F & = \sum_{m=1}^\ell \int_0^1 \frac{1}{(\tau\eps)^{\ell-m}} \tau^{\ell-1} D^w_m f(p \cdot \delta_{\tau\eps}(w)) \,d\tau -\frac{1}{\eps^\ell} \sum_{k=1}^\ell \eps^k T^k_p f(w)\\
     \overset{\eqref{eq_difference_quotients_notation}}&{=} \sum_{m=1}^\ell \int_0^1 \tau^{\ell-1} R^w_{\ell-m,\tau\eps}(D^w_m f)(p) \,d\tau + \underbrace{\sum_{m=1}^\ell \frac{1}{\eps^{\ell-m}} \int_0^1 \tau^{\ell-1} P_{\ell-m}(D^w_m f)(p,\delta_{\tau\eps}(w)) \,d\tau - \frac{1}{\eps^\ell} \sum_{k=1}^\ell \eps^k T^k_p f(w)}_{\eqqcolon G}.
\end{split}
\end{equation*}
Hence, it suffices to prove that $G=0$. Using the convention $T^0_p (D^w_m f)(w)=(D^w_m f)(p)$, we get
\begin{equation*}
\begin{split}
    G \overset{\eqref{eq_Taylor_polynomial}}&{=} \frac{1}{\eps^\ell}\left[\sum_{m=1}^\ell \int_0^1 \tau^{m-1} \eps^m \sum_{h=0}^{\ell-m} (\tau\eps)^h T^h_p(D^w_m f)(w) \,d\tau - \sum_{k=1}^\ell \eps^k T^k_p f(w)\right]\\
    & = \frac{1}{\eps^\ell}\left[\sum_{m=1}^\ell \sum_{h=0}^{\ell-m} \frac{\eps^{m+h}}{m+h} T^h_p(D^w_m f)(w) - \sum_{k=1}^\ell \eps^k T^k_p f(w)\right]\\
    & = \frac{1}{\eps^\ell}\left[\sum_{m=1}^\ell \sum_{k=m}^\ell \frac{\eps^k}{k} T^{k-m}_p(D^w_m f)(w) - \sum_{k=1}^\ell \eps^k T^k_p f(w)\right]\\
    & = \frac{1}{\eps^\ell}\sum_{k=1}^\ell \eps^k \left[\frac{1}{k} \sum_{m=1}^k T^{k-m}_p(D^w_m f)(w)-T^k_p f(w)\right].
\end{split}
\end{equation*}
We claim that the term inside the parentheses is $0$. Indeed, by the general Leibniz rule,
\begin{equation*}
\begin{split}
    T^k_p f(w) \overset{\eqref{eq_higher_order_derivatives},\eqref{eq_Taylor_coefficients}}&{=} \frac{1}{k!} \frac{d^k}{d\tau^k}\bigg\rvert_{\tau=0} f(p \cdot \delta_\tau(w))\\
    \overset{\eqref{eq_implicit_derivative}}&{=} \frac{1}{k!} \sum_{m=1}^s \frac{d^{k-1}}{d\tau^{k-1}}\bigg\rvert_{\tau=0} \tau^{m-1} D^w_m f(p \cdot \delta_\tau(w))\\
    & = \frac{1}{k!} \sum_{m=1}^s \sum_{h=0}^{\min\{k-1,m-1\}} \binom{k-1}{h} \frac{d^h}{d\tau^h}\bigg\rvert_{\tau=0} \tau^{m-1} \frac{d^{k-h-1}}{d\tau^{k-h-1}}\bigg\rvert_{\tau=0} D^w_m f(p \cdot \delta_\tau(w))\\
    \overset{\eqref{eq_Taylor_coefficients}}&{=} \frac{1}{k!} \sum_{m=1}^s \sum_{h=0}^{\min\{k-1,m-1\}} \binom{k-1}{h} \frac{(m-1)!}{(m-1-h)!} \tau^{m-1-h}\rvert_{\tau=0} (k-h-1)! T^{k-h-1}_p(D^w_m f)(w)\\
    & = \frac{1}{k!} \sum_{m=1}^k \binom{k-1}{m-1}(m-1)!(k-m)! T^{k-m}_p(D^w_m f)(w)\\
    & = \frac{1}{k} \sum_{m=1}^k T^{k-m}_p(D^w_m f)(w).
\end{split}
\end{equation*}
This shows that $G=0$, so the proof is concluded.
\end{proof}
Second, the boundedness of difference quotients yields a quantitative mollification error estimate. For simplicity, we consider global horizontal Sobolev functions, for which \Cref{thm_difference_quotients} holds in the following form.
\begin{theorem} \label{thm_difference_quotients_global}
Let $\ell \geq 1$ and $\theta \in [1,\infty)$. Let $f \in W_\hhh^{\ell,\theta}(\mathbb G)$. Then
\begin{equation*}\label{eq:priori10_lavorocarnot_global}
\lim_{\eps \searrow 0}^{L^\theta(\mathbb G)}  R^w_{\ell,\eps} f(p) = 0\qquad\text{for every $w \in \mathbb G$.}
\end{equation*}
In addition, there exists $C_2=C_2(\mathbb G,\ell)>0$ such that, for every $\varepsilon>0$ and $w\in\mathbb G$,
\begin{equation}\label{eq:priori11_lavorocarnot_global}
   \left\| R^w_{\ell,\eps}f\right\|_{L^\theta(\mathbb G)}\leq C_2\, d(w,0)^\ell\sum_{\alpha_1,\ldots,\alpha_\ell=1}^{n_1} \left\|X^1_{\alpha_1}\ldots X^1_{\alpha_\ell}f \right\|_{L^\theta(\mathbb G)}.
\end{equation}
Finally, if $f\in W^{\ell,\infty}_\hhh(\mathbb G)$, one still has 
\begin{equation}\label{eq:priori11_infinito_lavorocarnot_global}
   \left\| R^w_{\ell,\eps}f\right\|_{L^\infty(\mathbb G)}\leq C_2\, d(w,0)^\ell\sum_{\alpha_1,\ldots,\alpha_\ell=1}^{n_1} \left\|X^1_{\alpha_1}\ldots X^1_{\alpha_\ell}f \right\|_{L^\infty(\mathbb G)}.
\end{equation}
\end{theorem}
A simple consequence of \Cref{thm_difference_quotients_global} is the following quantitative version of \Cref{prop_group_mollification}.
\begin{theorem}
    Let $\ell \geq 1$ and $\theta \in [1,\infty]$. Let $f \in W_\hhh^{\ell,\theta}(\mathbb G)$. 
    Let $\rho \in C_c^\infty(B(0,1))$ be such that
\begin{equation}\label{eq_mollificatori_quant}
\int_{\mathbb{G}} \rho\,dp = 1,\qquad \text{and}\qquad\int_{\mathbb G}\rho\,P\,dp=0 
\end{equation}
for every non-constant polynomial $P$ of homogeneous degree less than or equal to $\ell-1$ and such that $P(0)=0$. Define $\rho_\eps$ as in \eqref{eq_def_mollifier}, and set $u_\eps\coloneqq u*\rho_\eps$. Then, there exists $C_3=C_3(\mathbb G,\ell)>0$ such that, for every $\varepsilon>0$,
\begin{equation*}
    \|u_\eps-u\|_{L^{\theta}(\mathbb G)}\leq C_3\,\eps^\ell \sum_{\alpha_1,\ldots,\alpha_\ell=1}^{n_1} \left\|X^1_{\alpha_1}\ldots X^1_{\alpha_\ell} u\right\|_{L^\theta(\mathbb G)}.
\end{equation*}
\end{theorem}
\begin{proof}
    Fix $p\in\mathbb G$. We have
    \begin{equation*}
        u_\eps(p)-u(p)\overset{\eqref{eq_dilation_Lebesgue},\eqref{eq_mollificatori_quant}}{=}\int_{\mathbb G}\left[u\left(p\cdot\delta_\eps\left(q^{-1}\right)\right)-u(p)\right]\rho(q)\,dq \overset{\eqref{eq_difference_quotients_notation},\eqref{eq_mollificatori_quant}}{=}\eps^\ell\int_{\mathbb G}\left[R^{q^{-1}}_{\ell,\eps}u(p)+T^\ell_p u\left(q^{-1}\right)\right]\rho(q)\,dq.
    \end{equation*}
    Therefore,
        \begin{equation*}
        \|u_\eps-u\|_{L^\theta(\mathbb G)}\leq\eps^\ell\int_{\mathbb G}\left(\left\|R^{q^{-1}}_{\ell,\eps}u\right\|_{L^\theta(\mathbb G)}+\left\|T^\ell_\cdot u\left(q^{-1}\right)\right\|_{L^\theta(\mathbb G)}\right)|\rho(q)|\,dq.
    \end{equation*}
    A simple estimate of $T^\ell_\cdot u$, combined with \eqref{eq:priori11_lavorocarnot_global} and \eqref{eq:priori11_infinito_lavorocarnot_global}, concludes the proof.
\end{proof}

\bibliographystyle{abbrv}
\bibliography{biblio}

\begin{thebibliography}{10}

\bibitem{MR2096794}
L.~Ambrosio.
\newblock Transport equation and {C}auchy problem for {$BV$} vector fields.
\newblock {\em Invent. Math.}, 158(2):227--260, 2004.

\bibitem{MR2409676}
L.~Ambrosio and G.~Crippa.
\newblock Existence, uniqueness, stability and differentiability properties of the flow associated to weakly differentiable vector fields.
\newblock In {\em Transport equations and multi-{D} hyperbolic conservation laws}, volume~5 of {\em Lect. Notes Unione Mat. Ital.}, pages 3--57. Springer, Berlin, 2008.

\bibitem{MR3283066}
L.~Ambrosio and G.~Crippa.
\newblock Continuity equations and {ODE} flows with non-smooth velocity.
\newblock {\em Proc. Roy. Soc. Edinburgh Sect. A}, 144(6):1191--1244, 2014.

\bibitem{MR2401600}
L.~Ambrosio, N.~Gigli, and G.~Savar\'e.
\newblock {\em Gradient flows in metric spaces and in the space of probability measures}.
\newblock Lectures in Mathematics ETH Z\"urich. Birkh\"auser Verlag, Basel, second edition, 2008.

\bibitem{dpl1}
L.~Ambrosio, G.~Somma, S.~Verzellesi, and D.~Vittone.
\newblock Renormalization of contact vector fields with horizontal {S}obolev regularity in {H}eisenberg groups.
\newblock Preprint, \url{https://doi.org/10.48550/arXiv.2602.00804}, 2026.

\bibitem{MR3265963}
L.~Ambrosio and D.~Trevisan.
\newblock Well-posedness of {L}agrangian flows and continuity equations in metric measure spaces.
\newblock {\em Anal. PDE}, 7(5):1179--1234, 2014.

\bibitem{MR4071413}
S.~Bianchini and P.~Bonicatto.
\newblock A uniqueness result for the decomposition of vector fields in {$\Bbb R^d$}.
\newblock {\em Invent. Math.}, 220(1):255--393, 2020.

\bibitem{BonLanUgu}
A.~Bonfiglioli, E.~Lanconelli, and F.~Uguzzoni.
\newblock {\em Stratified {L}ie groups and potential theory for their sub-{L}aplacians}.
\newblock Springer Monographs in Mathematics. Springer, Berlin, 2007.

\bibitem{MR4544986}
M.~Bramanti and L.~Brandolini.
\newblock {\em H\"ormander operators}.
\newblock World Scientific Publishing Co. Pte. Ltd., Hackensack, NJ, 2023.

\bibitem{MR4801826}
L.~Capogna, G.~Giovannardi, A.~Pinamonti, and S.~Verzellesi.
\newblock The asymptotic {$p$}-{P}oisson equation as {$p\to\infty$} in {C}arnot-{C}arath\'eodory spaces.
\newblock {\em Math. Ann.}, 390(2):2113--2153, 2024.

\bibitem{MR1411988}
I.~Capuzzo~Dolcetta and B.~Perthame.
\newblock On some analogy between different approaches to first order {PDE}'s with nonsmooth coefficients.
\newblock {\em Adv. Math. Sci. Appl.}, 6(2):689--703, 1996.

\bibitem{MR2369485}
G.~Crippa and C.~De~Lellis.
\newblock Estimates and regularity results for the {D}i{P}erna-{L}ions flow.
\newblock {\em J. Reine Angew. Math.}, 616:15--46, 2008.

\bibitem{AST_2008__317__175_0}
C.~De~Lellis.
\newblock Ordinary differential equations with rough coefficients and the renormalization theorem of {Ambrosio} [after {Ambrosio,} {DiPerna,} {Lions]}.
\newblock In {\em S\'eminaire Bourbaki - Volume 2006/2007 - Expos\'es 967-981}, number 317 in Ast\'erisque, pages 175--203. Soci\'et\'e math\'ematique de France, 2008.
\newblock talk:972.

\bibitem{DePauw_03}
N.~Depauw.
\newblock Non unicit\'e{} des solutions born\'ees pour un champ de vecteurs {BV} en dehors d'un hyperplan.
\newblock {\em C. R. Math. Acad. Sci. Paris}, 337(4):249--252, 2003.

\bibitem{MR1022305}
R.~J. DiPerna and P.-L. Lions.
\newblock Ordinary differential equations, transport theory and {S}obolev spaces.
\newblock {\em Invent. Math.}, 98(3):511--547, 1989.

\bibitem{MR494315}
G.~B. Folland.
\newblock Subelliptic estimates and function spaces on nilpotent {L}ie groups.
\newblock {\em Ark. Mat.}, 13(2):161--207, 1975.

\bibitem{MR0657581}
G.~B. Folland and E.~M. Stein.
\newblock {\em Hardy spaces on homogeneous groups}, volume~28 of {\em Mathematical Notes}.
\newblock Princeton University Press, Princeton, NJ; University of Tokyo Press, Tokyo, 1982.

\bibitem{MR1404326}
N.~Garofalo and D.-M. Nhieu.
\newblock Isoperimetric and {S}obolev inequalities for {C}arnot-{C}arath\'{e}odory spaces and the existence of minimal surfaces.
\newblock {\em Comm. Pure Appl. Math.}, 49(10):1081--1144, 1996.

\bibitem{MR1631642}
N.~Garofalo and D.-M. Nhieu.
\newblock Lipschitz continuity, global smooth approximations and extension theorems for {S}obolev functions in {C}arnot-{C}arath\'{e}odory spaces.
\newblock {\em J. Anal. Math.}, 74:67--97, 1998.

\bibitem{MR2737390}
P.-E. Jabin.
\newblock Differential equations with singular fields.
\newblock {\em J. Math. Pures Appl. (9)}, 94(6):597--621, 2010.

\bibitem{MR1920389}
A.~W. Knapp.
\newblock {\em Lie groups beyond an introduction}, volume 140 of {\em Progress in Mathematics}.
\newblock Birkh\"auser Boston, Inc., Boston, MA, second edition, 2002.

\bibitem{MR788413}
A.~Kor\'anyi and H.~M. Reimann.
\newblock Quasiconformal mappings on the {H}eisenberg group.
\newblock {\em Invent. Math.}, 80(2):309--338, 1985.

\bibitem{MR1317384}
A.~Kor\'anyi and H.~M. Reimann.
\newblock Foundations for the theory of quasiconformal mappings on the {H}eisenberg group.
\newblock {\em Adv. Math.}, 111(1):1--87, 1995.

\bibitem{MR2044334}
C.~Le~Bris and P.-L. Lions.
\newblock Renormalized solutions of some transport equations with partially {$W^{1,1}$} velocities and applications.
\newblock {\em Ann. Mat. Pura Appl. (4)}, 183(1):97--130, 2004.

\bibitem{MR2954043}
J.~M. Lee.
\newblock {\em Introduction to smooth manifolds}, volume 218 of {\em Graduate Texts in Mathematics}.
\newblock Springer, New York, second edition, 2013.

\bibitem{MR3726909}
G.~Leoni.
\newblock {\em A first course in {S}obolev spaces}, volume 181 of {\em Graduate Studies in Mathematics}.
\newblock American Mathematical Society, Providence, RI, second edition, 2017.

\bibitem{MR2124585}
N.~Lerner.
\newblock Transport equations with partially {$BV$} velocities.
\newblock {\em Ann. Sc. Norm. Super. Pisa Cl. Sci. (5)}, 3(4):681--703, 2004.

\bibitem{Libermann}
P.~Libermann.
\newblock Sur les automorphismes infinit\'esimaux des structures symplectiques et des structures de contact.
\newblock In {\em Colloque {G}\'eom. {D}iff. {G}lobale ({B}ruxelles, 1958)}, pages 37--59. Librairie Universitaire, Louvain, 1959.

\bibitem{MR4242824}
Q.-H. Nguyen.
\newblock Quantitative estimates for regular {L}agrangian flows with {$BV$} vector fields.
\newblock {\em Comm. Pure Appl. Math.}, 74(6):1129--1192, 2021.

\bibitem{MR2395129}
A.~Ottazzi.
\newblock A sufficient condition for nonrigidity of {C}arnot groups.
\newblock {\em Math. Z.}, 259(3):617--629, 2008.

\bibitem{MR2917692}
A.~Ottazzi and B.~Warhurst.
\newblock Contact and 1-quasiconformal maps on {C}arnot groups.
\newblock {\em J. Lie Theory}, 21(4):787--811, 2011.

\bibitem{MR4581339}
A.~Pinamonti, S.~Verzellesi, and C.~Wang.
\newblock The {A}ronsson equation for absolute minimizers of supremal functionals in {C}arnot-{C}arath\'{e}odory spaces.
\newblock {\em Bull. Lond. Math. Soc.}, 55(2):998--1018, 2023.

\bibitem{MR3587666}
F.~Serra~Cassano.
\newblock Some topics of geometric measure theory in {C}arnot groups.
\newblock In {\em Geometry, analysis and dynamics on sub-{R}iemannian manifolds. {V}ol. 1}, EMS Ser. Lect. Math., pages 1--121. Eur. Math. Soc., Z\"{u}rich, 2016.

\bibitem{MR266258}
N.~Tanaka.
\newblock On differential systems, graded {L}ie algebras and pseudogroups.
\newblock {\em J. Math. Kyoto Univ.}, 10:1--82, 1970.

\bibitem{MR746308}
V.~S. Varadarajan.
\newblock {\em Lie groups, {L}ie algebras, and their representations}, volume 102 of {\em Graduate Texts in Mathematics}.
\newblock Springer-Verlag, New York, 1984.
\newblock Reprint of the 1974 edition.

\bibitem{MR4986764}
S.~Verzellesi.
\newblock Variational properties of local functionals driven by arbitrary anisotropies.
\newblock {\em Calc. Var. Partial Differential Equations}, 65(1):Paper No. 3, 24, 2026.

\bibitem{MR2016308}
B.~Warhurst.
\newblock Contact and quasiconformal mappings on real model filiform groups.
\newblock {\em Bull. Austral. Math. Soc.}, 68(2):329--343, 2003.

\end{thebibliography}

\end{document}